\documentclass{article}
\usepackage{graphicx}
\usepackage{circuitikz}
\usepackage{amssymb}
\usepackage{answers}
\usepackage{setspace}
\usepackage{graphicx}
\usepackage{multicol}
\usepackage{mathrsfs}
\usepackage{hyperref}
\usepackage{amsmath,amsthm,amssymb}
\usepackage{tikz}
\usepackage{quiver}
\usepackage{natbib}
\usepackage{array, booktabs}

\usetikzlibrary {arrows.meta,bending,positioning, fit}

 \newcommand{\bbM}{\mathbb{M}}
 
 \newcommand{\bbQ}{\mathbb{Q}}

 \newcommand{\calT}{\mathcal{T}}
 
 \newcommand{\calP}{\mathcal{P}}

 \newcommand{\calC}{\mathcal{C}}

\newcommand{\calI}{\mathcal{I}}
\newcommand{\calJ}{\mathcal{J}}

\newcommand{\tp}{\text{tp}}

\newcommand{\qftp}{\text{qftp}}

\newcommand{\TP}{\text{TP}}
\newcommand{\NSOP}{\text{NSOP}}

\newcommand{\SOP}{\text{SOP}}
\newcommand{\SOPn}{$\text{SOP}_n$}
\newcommand{\TPi}{\text{TP}_1}
\newcommand{\TPii}{\text{TP}_2}

\newcommand{\IFF}{\enspace\text{if and only if}\enspace}

 \newcommand{\vphi}{\varphi}

 \newtheorem{claim}{Claim}
 \newtheorem{question}{Question}
 \newtheorem{fact}{Fact}
 \newtheorem{theorem}{Theorem}

\newtheorem{corollary}{Corollary}
\newtheorem{proposition}{Proposition}

\theoremstyle{definition}
\newtheorem{definition}{Definition}

\theoremstyle{definition}

\theoremstyle{definition}
 \newtheorem{remark}{Remark}

\def\Ind#1#2{#1\setbox0=\hbox{$#1x$}\kern\wd0\hbox to 0pt{\hss$#1\mid$\hss}
\lower.9\ht0\hbox to 0pt{\hss$#1\smile$\hss}\kern\wd0}

\def\ind{\mathop{\mathpalette\Ind{}}}

\def\notind#1#2{#1\setbox0=\hbox{$#1x$}\kern\wd0
\hbox to 0pt{\mathchardef\nn=12854\hss$#1\nn$\kern1.4\wd0\hss}
\hbox to 0pt{\hss$#1\mid$\hss}\lower.9\ht0 \hbox to 0pt{\hss$#1\smile$\hss}\kern\wd0}

\title{On the notion of a patterning property in model theory}
\author{Gabriel Day and Scott Mutchnik\footnote{Scott Mutchnik was supported by the NSF under Grant No. DMS-2303034.}}
\date{June 2026}

\begin{document}

\maketitle

\begin{abstract}

Different kinds of definable patterns in the models of a first-order theory, such as the order property, the tree property, or the ($n$-)strict order property, allow us to distinguish theories according to their logical complexity. The complexity distinctions given by these definable patterns play a central role in model theory. However, a rigorous definition of the notion of a model-theoretic patterning property has yet to be established.

We start by discussing different possible proposals from the literature for making the notion of a model-theoretic patterning property rigorous. Some examples will include the straight definability of \cite{shelah_what_2000}, which will describe properties definable by a pattern of consistency and inconsistency in a formula and its negation, and the poset definability of \cite{garcia2022model}, covering properties definable by interpreting a partial order embedding a given poset. We will also introduce a higher-arity version of straight definability.

In our first main result, we will answer open questions of \cite{bailetti2024walk} and \cite{garcia2022model}, showing that the $n$-strict order property $\mathrm{SOP}_{n}$ is straightly definable and poset definable even for integers $n \geq 4$. This will complete the categorization of all of the classical classification-theoretic properties as straightly definable.

Our other main result will concern properties that are straightly definable without negation: the positively straightly definable properties defined by \cite{bailetti2024walk}. We will show using Saracino's theorem and results of \cite{bodirsky2025taking} that, in any countably categorical theory, implications between positively straightly definable properties must be exhibited at the level of $\exists\forall$-formulas. This will have special consequences under the assumption that $\mathrm{SOP}_{2}$ is equal to $\mathrm{SOP}_{3}$.
    
\end{abstract}

\section{Introduction}

One of the major research programs in model theory is to identify and analyze complexity distinctions between first-order theories. These theories are formal descriptions of arbitrary mathematical structures, such as the real or complex numbers as a ring, the set-theoretic universe, and the Rado random graph. Early efforts in this subfield, \textit{classification theory}, concentrated on characterizing the tamer theories as those with smaller, more organized model classes.\footnote{Though ``simple" is the natural antonym to ``complex," its well-established use as a technical term in model theory, referring to theories which do not have the tree property, leads us to use ``tame" as the loose antonym for complex in this paper. This usage has a long history in model theory. 
} These investigations were quite successful, leading to research achievements such as Shelah's Main Gap Theorem and the isolation of stability as a central model-theoretic property.

However, classification theory has continued long past these early successes. Model theorists have continued to identify properties which, in some sense, separate tame from complex theories. But the underlying notion of complexity at play is not just that between theories with a few ``well-organized" models and those with a large, unclassifiable model class. A common feature to most of the distinctions identified to date is that there is some kind of definable structural pattern appearing in the models of complex theories which never appear in the models of tame theories.\footnote{See \cite{day2025complexity} for an explanation of the conceptions of complexity at play in classification theory, in particular distinguishing the classifiability and patterning senses of complexity.} In particular, these patterns are defined by a particular formula and are thus ``local" properties. The Order Property (OP), Independence Property (IP), and Tree Property (TP) are paradigm instances of such \textit{model-theoretic patterning properties}: properties of theories which indicate that they are complex in virtue of their instantiation of a specific definable structural pattern.


The abundance of these patterning properties leads naturally to the question of what a model-theoretic pattern is. Though there are many clear instances of patterning properties, like OP, IP, and TP, there has been, to this point, no universally agreed upon definition which characterizes all instances of such properties. In this paper, we catalogue several notions which have been used to characterize model-theoretic patterning properties, and evaluate their suitability as universal definitions of this notion. We support our evaluation of this question with several new results, including results answering open problems regarding the scope of these definitions of patterning.

Our central results show that all of the $n$-Strong Order Properties (\SOPn) can be characterized by two different notions of patterning property: that of a \textit{straightly definable property} and a \textit{poset definable property}. This answers open problems of Bailetti (\cite{bailetti2024walk}) and Garcia and Mennuni (\cite{garcia2022model}), respectively. Both of these notions offer plausible definitions of ``patterning property," though some notable open questions remain (see Section \ref{open questions}). 
In addition to settling the open questions on \SOPn, we offer a natural generalization of straight definability which accommodates higher arity patterning properties, such as $\text{IP}_k$ and $\text{FOP}_k$.

In addition to helping answer a fascinating conceptual question into the nature of model-theoretic patterning properties, studying these properties through a common definition allows for the development of results and common techniques which apply to a wide collection of patterning properties. Some of the central open problems of classification theory ask to determine, for particular straightly definable properties, whether those properties are equal or not. Among the subclass of \textit{positively straightly definable properties}, perhaps the most important question of this kind that still remains open (due to \cite{Dzamonja2004}, and mentioned as early as 1999 or 1997; see \cite{shelah_what_2000}), after one of the authors positively resolved the question of whether $\mathrm{SOP}_{1}$ is equal to $\mathrm{SOP}_{2}$ (\cite{Mutchnik2026nsop2}), is whether $\mathrm{SOP}_{2}$ is equal to $\mathrm{SOP}_{3}$. One last goal of this paper is to demonstrate a general phenomenon among equalities of positively straightly definable properties, or more generally, implications between positively straightly definable properties. As we will remark, in the result that every theory with $\mathrm{SOP}_{1}$ has $\mathrm{SOP}_{2}$, we get some control over this implication at the level of formulas: if $\varphi(x, y)$ is a formula that exhibits $\mathrm{SOP}_{1}$, there is a $\exists \forall$-formula in $\varphi(x, y)$ that exhibits $\mathrm{SOP}_{2}$.\footnote{In fact, in this specific case, a Boolean combination of existential formulas in $\varphi(x, y)$ and universal formulas in $\varphi(x, y)$.} We show that, for implications between general positively straightly definable properties, we get a similar form of control \textit{for countably categorical theories}. That is, for $P$ and $Q$ positively straightly definable properties, if every countably categorical theory with $P$ has $Q$, then in any countably categorical theory, if $\varphi(x, y)$ is a formula that exhibits $P$, there is a $\exists \forall$-formula in $\varphi(x, y)$ that exhibits $Q$. This general phenomenon is of particular interest for many specific open problems on implications between positively straightly definable properties. For example, if the answer to whether $\mathrm{SOP}_{3}$ is equal to $\mathrm{SOP}_{2}$ is yes, then  if $\varphi(x, y)$ is a formula that exhibits $\mathrm{SOP}_{2}$, there is a $\exists \forall$-formula in $\varphi(x, y)$ that exhibits $\mathrm{SOP}_{3}$; a similar result is true for the open question, posed by Palacín (\cite{palacin2012omega}), of whether every countably categorical simple theory is low.


\subsection{Outline}

This paper aims to achieve several goals, which may be of interest to distinct groups of readers. This outline should serve to direct readers to the sections of interest for their purposes.

First, in Section \ref{Patterning properties} we offer a self-contained introduction to several important  patterning properties: OP, IP, $\SOP_n$, and TP, in addition to the higher-arity properties $\text{IP}_k$ and $\text{FOP}_k$. These properties serve as indicators for whether or not an attempt to characterize patterning properties in general succeeds. In Section \ref{Definitions of patterning}, we survey the candidate definitions for ``model-theoretic patterning property": interpreting a structure, trace defining a structure,  and straight definition. This section compiles previously-known and new results (proved in Section \ref{New results on SOPn}) to compare the adequacy of these candidate definitions. We also generalize straight definability to the higher-arity context, to encompass patterning properties like $\mathrm{IP}_{k}$ and $\mathrm{FOP}_{k}$. In Section \ref{New results on SOPn} we prove several new results about straight definability and interpretation of posets, in particular, the characterization of $\SOP_n$ via both of these definitions. Finally, in Section \ref{positive definability categoricity}, we examine implications between positively straightly definable properties, showing that they must hold at the level of $\exists\forall$-formulas in countably categorical theories, and discussing the implications for open problems such as whether $\mathrm{SOP}_{2}$ is equal to $\mathrm{SOP}_{3}$. 

Readers interested solely in our new results may skip to Sections \ref{New results on SOPn} and \ref{positive definability categoricity}
referring to previous sections only for the statements of definitions. Readers interested solely in our analysis of how to define ``model-theoretic patterning property" may look to Section \ref{Patterning properties} and \ref{Definitions of patterning}; the chart at the end of Section \ref{New results on SOPn} summarizing which notions of patterning can define which patterns may also be of interest.

\section{Patterning Properties}\label{Patterning properties}

\subsection{Binary patterns}


Before considering different candidates for definitions of what a model-theoretic property is, we first introduce some of the main classification-theoretic properties that we might refer to as patterning properties. We take it as a datum that these properties are united informally by requiring some kind of definable pattern be instantiated in models oof a theory. Specifically, each of the properties we discuss below will be defined according to some combinatorial pattern in a formula in a first-order theory. The presence of such a pattern may be thought of as indicating that the theory contains ``more structure" and is thus, in a sense, more complex. This intuitive judgment is supported by the fact that many of these properties, like OP, IP, and TP, imply that particular niceness conditions fail. We will begin with stability, whose foundational role in classification theory is demonstrated, for instance, through Shelah's proof of the Main Gap Theorem. We will then turn to further properties originated by Shelah, with the goal of classifying the more complex, unstable theories.

The \textit{order property}, of which stability is defined to be the negation, essentially says that there is a formula that linearly orders an infinite set (though the property has several equivalent definitions, or which we offer two).

\begin{definition}
A theory $T$ is said to have the \textit{Order Property}  ($\mathrm{OP}$) if there is a formula $\vphi(x,y)$ in the language of $T$ and tuples $(a_i, b_i)_{i< \omega}$ in a model $M\models T$ such that $$M\models \vphi(a_i,b_j)\enspace\text{if and only if}\enspace i<j.$$ Equivalently, there is a formula $\varphi(x, y)$ with $|x|=|y|$ and a sequence of tuples $\{b_{i}\}_{i <\omega}$ such that $$M\models \vphi(b_i,b_j)\enspace\text{if and only if}\enspace i<j.$$ A theory without the order property is said to be \textit{stable}.

\end{definition}

The following diagram depicts the first two of these equivalent definitions of OP. Solid arrows between a pair of variable tuples indicate that the formula $\vphi$ holds between that pair, while dashed arrows indicates that the negation of $\vphi$ holds between those tuples. For instance, we have $\models\vphi(a_0, b_1)$, and $\models\neg\vphi(a_1,b_0)$.


\begin{center}
\begin{tikzpicture}
    \node(b1){$b_{0}$};
    \node(b2)[right of = b1]{$b_{1}$};
    \node(dots0)[right of =b2]{$\dotsm$};
    \node(empty)[right of =dots0]{$b_i$};
    \node(dots1)[right of =empty]{$\dotsm$};
    \node(satisfied)[below of =empty]{};
    \node (ai)[below of =satisfied]{$a_i$};
    \node (adots1)[left of =ai]{$\dotsm$};
    \node (adots2)[right of =ai]{$\dotsm$};
    \node (a2)[left of =adots1]{$a_1$};
    \node (a1)[left of =a2]{$a_0$};
    \draw[thick, ->] (a1) -- (b2);
    \draw[thick, ->] (a1) -- (empty);
    \draw[thick, ->] (a2) -- (empty);
    \draw[thick, dashed, ->](a2) -- (b1);
    \draw[thick, dashed, ->](ai) -- (b2);
    \draw[thick, dashed, ->](ai) -- (b1);
     
\end{tikzpicture}
\end{center}

The independence property is a natural extension of the order property. While the order property gives a formula selecting the initial segments of some infinite set, the independence property amounts to a combinatorial instance of randomness within a formula, giving a formula selecting \textit{all} of the subsets of some infinite set.

\begin{definition}
A theory $T$ has the \textit{independence property} (IP) if there is a formula $\vphi(x, y)$ in the language of $T$ with the following property: There is a model $M\models T$ such that there are tuples $\{a_i\colon i< \omega\}$ and $\{b_{s}\colon s\subseteq \omega\}$ from $M$ such that \[M\models \vphi(a_i, b_{s})\enspace\text{if and only if} \enspace i\in s.\] 

A theory without $\mathrm{IP}$ is said to be \textit{dependent} or $\mathrm{NIP}$: it satisfies the \textit{negation of the independence property}.



\end{definition}



In the following diagram, as in the previous one, solid lines connect pairs where $\vphi$ holds, while dashed lines indicate pairs where $\vphi$ fails to hold. These encode, respectively, when the indices of $a_i$ tuples are elements of the sets coded by indices of $b_s$ tuples, and when they are not elements of those sets. For instance, since $0\in\{0,i\}$, there is a solid arrow from $a_0$ to $b_{\{0,i\}}$. Meanwhile, since the empty set contains no elements, $1\notin \emptyset$, so there is a dashed arrow from $a_1$ to $b_\emptyset$.

\begin{center}
\begin{tikzpicture}
    \node(b1){$b_{\emptyset}$};
    \node(dots0)[right of =b1]{$\dotsm$};
    \node(empty)[right of =dots0]{$b_{\{0, i \}}$};
    \node(dots1)[right of =empty]{$\dotsm$};
    \node(bi)[right of =dots1]{$b_s$};
    \node (dots2)[right of =bi]{$\dotsm$};
    \node(satisfied)[above of =dots1]{};
    \node (ai)[above of =satisfied]{$a_i$};
    \node (adots1)[left of =ai]{$\dotsm$};
    \node (adots2)[right of =ai]{$\dotsm$};
    \node (a2)[left of =adots1]{$a_1$};
    \node (a1)[left of =a2]{$a_0$};
    \draw[thick, ->] (ai) -- (empty);
    \draw[thick, ->] (a1) -- (empty);
    \draw[dashed, ->](a2) -- (b1);
    \draw[dashed, ->](ai) -- (b1);
    \draw[dashed, ->](a1) -- (b1);

\end{tikzpicture}
\end{center}

Extending the order property in different way, the \textit{strict order property} asserts the existence of a formula defining a family of sets with an ascending chain under strict inclusion. As we will see in Section \ref{interpretation section}, this is equivalent to interpreting a partially ordered set with an infinite chain.

\begin{definition}
    A theory has the \textit{strict order property} (SOP) if there is a formula $\vphi(x,y)$ and a sequence of parameters $(b_i)_{i<\omega}$ such that, for each $i,j<\omega$, \[i<j \IFF \models \exists x(\vphi(x,b_{j})\land\neg\vphi(x,b_i)).\] Otherwise it is $\mathrm{NSOP}$.

    We can visualize some of the sets $\varphi(M, b_{i})$ defined by the formulas $\varphi(x, b_{i})$ as a series of expanding boxes, with $\varphi(M,b_i) \subsetneq \varphi(M,b_j) \iff i < j$.\footnote{The following illustration is drawn from John Baldwin's \textit{Fundamentals of Stability Theory}, \cite{baldwin1988fundamentals}, p90.}
\end{definition}

\includegraphics[scale=.65]{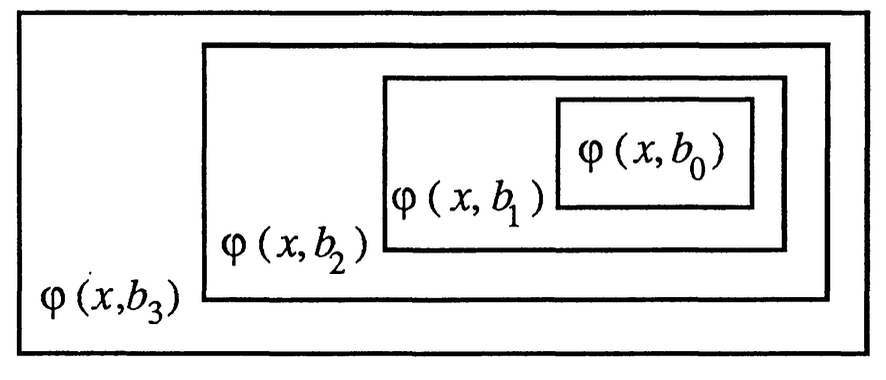}

Note that each of these properties is defined for a binary partitioned formula with two variable tuples $x$ and $y$; these may be labeled \textit{binary patterns}. We will see in Section \ref{higher arity pattern definitions section} that some patternings are defined using formulas with more variable partitions. But first, we will include the definitions of two more (classes of) binary patterns, the first of which, \SOPn, is the subject of our new results in Section \ref{New results on SOPn}, and the second of which, the tree property $\mathrm{TP}$, is the subject of a notable open question about one candidate patterning notion (poset definability), and rules out another (trace definability) as a universal definition.

\begin{definition}\label{definition of sopn}
    Let $n\geq 3$.
    A theory $T$ is said to have the $n$\textit{-Strong Order Property} (\SOPn) if there is a formula $\vphi(x,y)$ with $x$ and $y$ tuples of equal length, and a model $M\models T$ such that

    \begin{enumerate}
        \item The set of formulas $\{\vphi(x_1,x_2), \dotsm, \vphi(x_{n-1},x_n),\vphi(x_n,x_1)\}$
        is inconsistent with $T$;

        \item There is a sequence $(a_i)_{i<\omega}$ from $M$ such that, for all $i<j$, we have $M \models \vphi(a_i,a_j)$.
    \end{enumerate}

   If $T$ does not have $\SOP_n$, then it is said to be $\NSOP_n$.
\end{definition}

We will also give the definitions of $\mathrm{SOP}_{1}$ and $\mathrm{SOP}_{2}$, which are defined differently, in section 5 below.

As the name suggests, each $\mathrm{SOP}_n$ is a strengthening
of OP. A formula witnesses $\mathrm{SOP}_n$ for $T$ along an infinite sequence of elements which are ``linearly ordered" by $\vphi$, just as a formula witnessing OP does.\footnote{The definition of $\mathrm{SOP}_{n}$ includes the requirement that if $i<j$, then $M\models\vphi(a_i,a_j)$, as is required for OP. OP also requires, conversely, that if $i>j$, then $M\models \neg \vphi(a_i,a_j)$. While not stated immediately in the definition of $\SOP_n$, we may assume $\{a_{i}\}_{i < \omega}$ is indiscernible, in which case this property follows from it: if $i>j$ and $M\models \vphi(a_i,a_j)$, then the anti-looping clause (1) would be violated, since we may assume $i = j + (n-1)$, and then $\{\vphi(a_i,a_{i -1}), \dotsm, \vphi(a_{i- k},a_{i - (k + 1)}), \ldots, \vphi(a_j,a_i)\}$ would be satisfied by $M$.} But in addition, clause (1) guarantees that the relation defined by $\vphi$ cannot ``loop back" on itself in $n$ many steps.

The final binary patterning property we will mention, the Tree Property (TP), notably uses a slightly different definitional framework from those above. While OP, IP, and $\mathrm{SOP}_n$ are equivalent to the existence of a definable binary relation with certain properties on a specified set of elements, TP asserts that a particular family of definable sets exists, with a specific ``branching" structure reflecting the relation of a mathematical tree.

In order to define the tree property, we must first offer some notation and definitions for trees. A \textit{tree} is a set $\calT$ partially ordered by a relation $\trianglelefteq$ such that, for any element (or node) $\eta\in \calT$, the set $\{\nu\in \calT\colon \nu\trianglelefteq\eta\}$ is linearly ordered by $\trianglelefteq$. It is often helpful to view trees as collections of functions between ordinals. For instance, $\omega^{<\omega}$, the set of functions $\eta\colon n\to\omega$ for all finite ordinals $n$, is a tree of countable height whose nodes each branch countably. A node in this tree, a function $\eta$, \textit{extends} another, $\nu$, if $\nu$ is the restriction $\eta|_n$ of $\eta$ to some smaller ordinal $n$. When this is the case, we write $\nu\trianglelefteq\eta$. A \textit{path} through this tree is a function $\beta\colon \omega\to\omega$; any restriction $\beta|_n$ of a path to a finite ordinal $n$ is a node in the tree. For $\eta\in \omega^{<\omega}$ and $i<\omega$, $\eta^\smallfrown i$ denotes the function from $n+1\to \omega$ which agrees with $\eta$ on all $j\leq n$, and takes value $i$ on $n+1$.

\begin{definition}\label{tree property}
    A theory $T$ has the \textit{tree property} (TP) if there is a formula $\vphi(x,y)$ in the language of $T$ and a model $M\models T$ which contains tuples $(a_\eta)_{\omega^{<\omega}}$ satisfying the following:

    \begin{enumerate}
        \item For each path $\beta\in\omega^{\omega}$, the set of formulas $\{\vphi(x,a_{\beta|_n})\colon n<\omega\}$ is consistent.

        \item There is a natural number $k$ such that for each $\eta\in\omega^{<\omega}$, any $k$ elements of the set $\{\vphi(x,a_{\eta^\smallfrown \langle i \rangle})\colon i<\omega\}$ is inconsistent.
    \end{enumerate}

A formula $\vphi$ such as this is also said to have the \textit{tree property}. A theory which doesn't have TP is called \textit{simple}. 

In the below visualization, the paths correspond to consistent sets of formulas, while the sets of immediate successors, some of which are denoted by dashed boxes, correspond to sets of formulas any $k$ elements of which are inconsistent. 

    \begin{center}
\begin{tikzpicture}
    \node(a0){$a_0$};
    \node(a1)[right of = a0]{$a_1$};
    \node(a2)[right of =a1]{$a_2$};
    \node(dots1)[right of =a2]{...};
    \node(ai)[right of = dots1]{$a_i$};
    \node (dots2)[right of =ai]{...};
    \node(a01)[above of =a0]{$a_{01}$};
    \node(a00)[left of =a01]{$a_{00}$};
    \node(dots01)[right of =a01]{...};
    \node(ai1)[above of =ai]{$a_{i1}$};
    \node(ai0)[left of =ai1]{$a_{i0}$};
    \node(dotsi1)[right of =ai1]{...};
    \draw[ultra thick](a0) -- (a00);
    \draw[black](a0) -- (a01);
    \draw[black](a0) -- (dots01);
    \draw[black](ai) -- (ai0);
    \draw[black](ai) -- (ai1);
    \draw[black](ai) -- (dotsi1);
    \node(a001)[above of =a00]{$a_{001}$};
    \node(a000)[left of =a001]{$a_{000}$};
    \node(dots001)[right of =a001]{...};
    \draw[black](a00) -- (a000);
    \draw[ultra thick](a00) -- (a001);
    \draw[black](a00) -- (dots001);
    \node(ai01)[above of =ai0]{$a_{i01}$};
    \node(ai00)[left of =ai01]{$a_{i00}$};
    \node(dotsi01)[right of =ai01]{...};
    \draw[black](ai0) -- (ai00);
    \draw[black](ai0) -- (ai01);
    \draw[black](ai0) -- (dotsi01);
    \node(vdots1)[above of =a001]{$\vdots$};
    \node(vdots2)[above of =ai01]{$\vdots$};
    \draw[ultra thick, ->](a001) -- (vdots1); 
    \node[draw,dashed,fit=(ai00) (ai01) (dotsi01)] {};
    \node[draw,dashed,fit=(ai0) (dotsi1)] {};
    
\end{tikzpicture}
\end{center}

In fact, \cite{shelah1990classification} shows that this is equivalent to the tree property with $k = 2$, which is necessary to show that the tree property is (positively) straightly definable.

\end{definition}

\subsection{Higher arity patterns}\label{higher arity pattern definitions section}

All of the patterning properties discussed so far are defined for a partitioned formula with two arguments: the variable tuples $x$ and $y$ appearing in $\vphi(x,y)$ witnessing OP, IP, TP, etc. Yet at least since \cite{shelah2007definable}, some model theorists have investigated higher arity analogues of some of these properties; i.e., properties defined for formulas with more variable tuples which may be substituted independently from one another. Thus a formula $\vphi(x,y,z)$ might witness a higher arity patterning property, where both the variables $y$ and $z$ are substituted for parameters independently of one another. 

While these higher arity properties have not garnered as much scholarly attention as the paradigmatic properties discussed above, they are subject of increasing interest in the model theory and combinatorics literature (see, for a sample, \cite{hempel2016n}, \cite{chernikov2020hypergraph}, and \cite{terry2021higher}). Thus, as we discuss candidate definitions of patterning property below, it will be a superlative on a definition's success if it can capture these higher arity properties. To this end, we offer a natural generalization of straight definability, a successful notion for capturing binary patternings, which is able to capture higher arity patterns as well.

By far the most studied higher arity pattern is the following generalization of the independence property, known as $\text{IP}_k$.

\begin{definition}\label{IPk definition}
    For $k\geq 1$, A formula $\vphi(x, y_0, \dots, y_{k-1})$ has the $k$-\textit{Independence Property}, $\text{IP}_k$, if there are sequences of tuples $(a_{0,i})_{i<\omega},\dotsm, (a_{k-1,i})_{i< \omega}$ and tuples $(b_s)_{s\subseteq\omega^k}$ such that 

    \[\models \vphi(b_s, a_{0,i_0},\dotsm, a_{k-1,i_{k-1}})\IFF (i_0, \dotsm ,i_{k-1})\in s.\]

\end{definition}

Note that $\text{IP}_1$ is precisely IP. 

While $\text{IP}_k$ has certainly gotten the most attention of any higher arity patterning property, it is not the only one being studied. In \cite{abd2025higher}, the authors Abd Aldaim, Conant, and Terry introduced the \textit{functional order property}, a higher-arity analogue of OP.

\begin{definition}\label{FOPk defn}
    For $k\geq 1$, a formula $\vphi(x, y_0,\dotsm, y_{k-1})$ has the $k$-\textit{functional order property}, $\text{FOP}_k$, if there are sequences $(a_{0,i})_{i<\omega},\dotsm, (a_{k-1,i})_{i < \omega}$ and $(b_f)_{f\colon \omega^{k-1}\to\omega}$ such that 

    \[\models \vphi (b_f, a_{0,i_0},\dotsm, a_{k-1,i_{k-1}})\IFF i_{k-1}\leq f(i_0,\dotsm, i_{k-2}).\]

\end{definition}

\section{Definitions of patterning}\label{Definitions of patterning}

We have included a suite of patterning properties, both binary and higher arity. A central question of this paper is whether there is a precise mathematical definition which is able to capture all of them, showing these patterning properties to all be members of a common species. We begin by outlining two options demonstrating partial success, namely \textit{interpretation} and \textit{trace definition}, before turning to \textit{straight definition}, the candidate with the best claim to universality.

There are really two definitions needed in order to make the notion of a patterning property precise---first, what a pattern is, and second, what it means for a theory or model to ``include" or ``be patterned by" a given pattern, as witnessed by some formula. The candidate notions of interpretation and trace definition offer definitions of patterning property reliant on the assumption that patterns are first-order structures. The final option, straight definability, relies on a definition of patterns as combinatorial rules determining the intersections of definable sets. After presenting each of these definitions, we will explore their possibilities and limitations as precisifications of the notion of a patterning property.

\subsection{Interpretation}\label{interpretation section}

A first pass definition of a pattern is simply a structure (such as a linear order or random graph), the ``inclusion" of which in other structures indicates complexity. Temporarily, we will assume that by ``pattern" we simply mean structure (until Section \ref{straight definition section}). The most classical model-theoretic notion of ``including" one structure in another (of a potentially distinct language) is interpretation. 

When an $L$-structure $M$ interprets an $L'$-structure $N$, the domain of $N$ is defined in $M^k$ by an $L$-formula, and all of the pieces of language in $L'$ which give $N$ its structure are reproduced using defining $L$-formulas. More precisely, given these assumptions, we define interpretation as follows:

\begin{definition}\label{interpretation definition}
    An \textit{interpretation} of $N$ in $M$ is a tuple $\Gamma$ consisting of three items: 
    \begin{enumerate}
        \item An $L$-formula $\vphi_D(x_0,\dotsm, x_{k-1})$ defining a set which serves as the domain of $N$ in $M^k$, for some natural number $k\geq 1$; 

        \item For each atomic formula $\sigma(y_0,\dotsm, y_{m-1})$ of $L'$, an $L$-formula $\vphi_\sigma(\overline{x_0},\dotsm, \overline{x_{m-1}})$, where each $\overline{x_i}$ is a $k$-tuple;

        \item An injective map $\tau\colon N \to \vphi_D(M^k)/E$, for $E$ the equivalence relation on $\vphi_D(M^k)$ defined by $\varphi_{y_1=y_2}(\overline{x_1}, \overline{x_2})$. 
    \end{enumerate} 

These items must be such that, for each atomic $L'$-formula $\sigma$ and for all $a_0,\dotsm, a_{n-1}\in N$, $N\models \sigma(a_0,\dotsm, a_{n-1}) \IFF M\models \vphi_\sigma(\tilde{\tau(a_0)},\dotsm,\tilde{\tau(a_{n-1})})$, for $\tilde{\tau(a_{i})}$ any choice of representative of $\tau(a_{i})$.
\end{definition} 


At first glance, interpretation may appear to be an ineffective definition for patterning properties. OP, for instance, is not equivalent to interpreting an infinite linear order, because OP does not require that the domain set of tuples which are ``linearly ordered" by the formula witnessing OP be definable---the definition simply asserts that such a set exists. The random graph demonstrates this distinction, since it has OP, yet does not interpret an infinite linear order, since no domain set is definable. However, more nuanced uses of interpretation do capture many paradigmatic patterning properties.

In \cite{garcia2022model}, Garcia and Mennuni showed that many patterning properties can be characterized by interpretation, and in particular, by interpreting posets. A \textit{partially-ordered set}, or \textit{poset}, is a structure $(\Sigma, <)$, where the binary relation $<$ is irreflexive, asymmetric, and transitive on the domain set $\Sigma$. The first noted model theoretic patterning to be characterized by posets was the Strict Order Property (SOP). SOP is equivalent to the interpretation of a poset with an infinite chain. A poset having an infinite chain means that it embeds the poset $(\omega, <)$, where $<$ gives $\omega$ its standard strict linear order. This example gave the basis for Garcia and Mennuni's definition for $\Sigma$-SOP for an arbitrary poset $\Sigma$. 

\begin{definition}
    Suppose $(\Sigma, <)$ is a poset. Then a theory $T$ has $\Sigma$-SOP if there exists a partitioned formula $\vphi(x,y)$ and a set of tuples from $M\models T$, $\{a_s\colon s\in\Sigma\}$, where $|a_s|=|y|$ for all $s$, such that 
    \[M\models \forall x(\vphi(x,a_s)\to\vphi(x,a_t))\IFF s\leq t.\]

\end{definition}



As noted in \cite{garcia2022model}, $\Sigma$-SOP is trivially equivalent to some $M\models T$ interpreting a poset $\Sigma'$ into which $\Sigma$ embeds. Thus, in this terminology, $\SOP$ is equivalent to $(\omega, <)-\SOP$.


However, as Garcia and Mennuni show in their paper, a host of other classificatory properties are equivalent to $\Sigma$-SOP for a carefully-chosen $\Sigma$. 

\begin{definition}\label{poset definable defn}
 A property $P$ of formulas (respectively, theories) is said to be \textit{poset definable} if there is a poset $\Sigma_P$ such that $P$ holds of a formula (respectively, a theory) precisely when it has $\Sigma_P$-SOP. 
\end{definition}

Garcia and Mennuni explicitly build posets witnessing that OP and IP are poset definable. Then, they go on to show that a broad class of properties, the \textit{tree properties} arising from \textit{maximal consistency patterns}, are each poset definable. We will return to these maximal consistency patterns in Section \ref{straight definition section}, as they are closely related to straight definability, another important definition for a patterning properties. But for now, it suffices to note that many known patterning properties, like $\TP_1$, $\TP_2$, $\text{ATP}$, and $\SOP_3$ all fall under this heading, and are thus shown to be poset definable. In Section \ref{New results on SOPn}, we show that each $\mathrm{SOP}_n$  property is poset definable.

However, a notable open question remains for $\Sigma$-SOP as a universal definition of patterning: whether TP is poset definable. Since the consistency pattern giving rise to TP is not maximal (unlike its extremal forms, $\TPi$ and $\TPii$), Garcia and Mennuni's results do not determine whether a poset exists to characterize TP. This potential  omission also extends to the class of generalized tree properties defined in \cite{day2025results}, the properties $\calI$-TP for a Ramsey index structure $\calI$. 


Additionally, Garcia and Mennuni observe that, as a consequence of results in \cite{chernikov2019n}, it is impossible to characterize $\text{IP}_k$, for $k\geq 2$, with posets (\cite{garcia2022model}, Remark 1.2). However, as they point out in Question 4.5, there may be some other structure, perhaps analogous to posets, such that $\text{IP}_k$ can be characterized using interpretations of these structures. This question is, to our knowledge, still open.

\subsection{Trace definition and collapse of indiscernibles}

Recall that OP is not equivalent to interpreting an infinite linear order because it does not require that the order's domain be definable. The notion of \textit{trace definition}, offered in \cite{walsberg2025trace}, removes this requirement that the domain be definable, and consequently gives another plausible definition of patterning property. Essentially, an $L$-structure $M$ trace defines an $L'$-structure $N$ if it defines all of the pieces of language in $L'$ giving $N$ its structure, but it doesn't necessarily define the domain on $N$.
\begin{definition}\label{trace defn defn}
     Let $\tau\colon N\to M^k$ be an injective function for some $k \geq 1$. Then $M$ \textit{trace defines} $N$ via $\tau$ if, for each $L'$-formula $\sigma(x_0,\dotsm, x_{n-1})$, there is an $L$-formula $\vphi_\sigma$ such that, for all $a_0,\dotsm, a_{n-1}\in N$, $N\models \sigma(a_0,\dotsm, a_{n-1}) \IFF M\models \vphi_\sigma(\tau(a_0),\dotsm,\tau(a_{n-1})).$
\end{definition}
We say that a theory $T$ \textit{trace defines} a structure $N$ if some model $M$ of $T$ trace defines $N$, and $T$ trace defines another theory $T'$ if it trace defines each model of $T'$.

It should be noted that Walsberg does not introduce trace definability for the purpose of giving a common way of defining every model-theoretic patterning property of interest. Rather, his purpose is to introduce a particularly rich class of potentially higher-arity patterns generalizing the independence property. One may nonetheless be interested in seeing which patterning properties are trace definable. Walsberg showed that OP and IP can be characterized in terms of trace definability. A theory has OP if, and only if, it trace defines the structure $(\bbQ,<)$, and a theory has IP if, and only if, that theory trace defines the random graph. 

In spite of this promising start, trace definability falls short of characterizing all patterning properties. For instance, TP is not equivalent to the trace definability of any structure. Walsberg proved in \cite{walsberg2025trace} that there are simple theories which are \textit{trace maximal}, i.e., which trace define every structure (in a countable language). For OP and IP, theories on the complex, ``pattern-including" side of the dichotomy trace define some particular structure, while those on the simple, ``pattern-excluding" side fails to define that structure. Here, TP is the complex side, and yet some theories which omit TP can trace define every possible structure. If trace defining a structure is what we meant by including a pattern, then, were $\mathrm{TP}$ characterizable in terms of trace definability, it would be impossible to have trace maximal theories which don't include the tree property pattern. So it is impossible to find a trace definition characterizing TP.

Trace definability fares better at characterizing higher-arity patterns. To see why, it is helpful to introduce the notion of generalized indiscernible sequences, in order to discuss so-called ``indiscernible collapse" characterizations of model-theoretic properties. 

\begin{definition}{}
Suppose $\calI$ is an $L'$-structure and $M$ is an $L$-structure. Let $A=(a_i)_{i\in \calI}$ be a collection of tuples from $M$. Then we say that the $A$ is a \textit{generalized indiscernible (set)}, and in particular, an $\calI$\textit{-indexed indiscernible (set)} if, for all $n\geq 1$ and $i_0,i_1,\dotsm, i_{n-1},j_0,j_1,\dotsm, j_{n-1}\in \calI$, if $\qftp_\calI(i_0,\dotsm, i_{n-1})=\qftp_\calI(j_0,\dotsm, j_{n-1})$, then $$\tp_M(a_{i_0},\dotsm a_{i_{n-1}})=\tp_M(a_{j_0},\dotsm, a_{j_{n-1}}).$$
\end{definition}

When $\calI$ is taken to be an infinite linear order, then the generalized indiscernible is called an \textit{indiscernible sequence}. This is the original and most widely studied kind of indiscernible. When $\calJ$ is a structure expanding a linear order, $\calJ$-indexed indiscernibles are said to \textit{collapse} to indiscernible sequences in models of a theory $T$ if, whenever $(a_i)_{i\in\calJ}$ is a set of $\calJ$-indexed indiscernibles, it is also an indiscernible sequence.

There is a close connection between trace definability and (the failure of) indiscernible collapse. IP theories trace define the random graph, but they are also precisely the theories in which ordered random graph indiscernibles do not collapse, by \cite{scow2012characterization}. The trace definition of OP can also be seen as a failure to collapse, based on Shelah's result that indiscernible sequences are not always indiscernible sets precisely in the theories with OP.\footnote{Strictly speaking, this characterization of stability does not meet our definition of collapse, since we define a generalized indiscernible as collapsing \textit{to} an indiscernible sequence, while Shelah shows that in stable theories, indiscernible sequences collapse \textit{to} indiscernible sets, an even stronger notion of indiscernibility. But morally, this is very much an indiscernible collapse result.} Walsberg makes the connection precise (\cite{walsberg2025trace} Proposition 9.30), showing that a theory $T$ trace defines a finitely homogeneous Ramsey structure $\calJ$ if and only if there is some non-collapsed indiscernible copy of $\calJ$ in a model of $T$.

In \cite{chernikov2019n}, the authors extend Scow's result for IP by showing (Theorem 5.4) that a theory has $\text{IP}_k$ if and only if there is some generically-ordered $k+1$-ary hypergraph indiscernible which is not indiscernible, i.e., if these hypergraph indiscernibles fail to collapse. Meanwhile, $\text{FOP}_k$ is also equivalent to the failure of a particular indiscernible structure to collapse by results of \cite{abd2025higher}; the bulk of that paper consists of constructing that generalized indiscernible index structure, so its statement is beyond our scope here. But in summary, both of the higher arity classificatory properties we consider can be captured by collapse of indiscernibles, and hence by trace definition.

\subsection{Straight definition}\label{straight definition section}

The final  definition of model-theoretic patterning property which we discuss requires us to drop the assumption that patterns are first-order structures. Instead, we take them to be sets of indices which determine the intersections of definable sets: \textit{consistency} and \textit{inconsistency conditions}. The definition of TP provides a good example of this characterization, where the consistency conditions are the assertion that paths in the indexing tree give consistent sets of formulas, while inconsistency conditions determine that siblings in the tree give inconsistent definable sets. Slightly modifying the terminology from \cite{bailetti2024walk}, we present the notion of a \textit{straightly definable property}, originally defined in \cite{shelah_what_2000}. 

\begin{definition}\label{straight pattern definition}
     A \textit{straight pattern} is a pair $(\calC,\calI)$, where $\calC$ and $\calI$ are both subsets of $(\calP(\omega)\times\calP(\omega))\setminus\{(\emptyset,\emptyset)\}$.\footnote{In \cite{bailetti2024walk}, what we call straight patterns were defined as $n$-patterns, their finite approximations. By compactness, this distinction is insignificant, and this way of defining things makes certain proofs slightly neater, specifically that of Proposition \ref{poset implies straight}.}

    A pair of sets $(X^+,X^-)\in \calC$ is a \textit{consistency condition} of the straight pattern $(\calC,\calI)$. 
    
    A pair of sets $(Y^+,Y^-)\in \calI$ is an \textit{inconsistency condition} of $(\calC,\calI)$.
    \end{definition}

\begin{definition}\label{exhibiting a straight pattern}
    A formula $\vphi(x,y)$ \textit{exhibits} a straight pattern $(\calC,\calI)$ in an $L$-theory $T$ if there is a set $B=(b_i)_{i<\omega}$ of $|y|$-length tuples from $M\models T$ such that, for all $(A^+,A^-)\in \calC$ and $(Z^+,Z^-)\in \calI$, we have 
    \[\{\vphi(x,b_i)\colon i\in A^+\}\cup \{\neg\vphi(x,b_j)\colon j\in A^-\}\] is consistent, and \[\{\vphi(x,b_i)\colon i\in Z^+\}\cup \{\neg\vphi(x,b_j)\colon j\in Z^-\}\] is inconsistent.

    A theory $T$ \textit{exhibits} a straight pattern $(\mathcal{C}, \mathcal{I})$ if some formula $\varphi(x, y)$ exhibits $(\mathcal{C}, \mathcal{I})$.
\end{definition}

Together, these notions yield a way of characterizing properties of formulas and theories.

\begin{definition}\label{straightly definable property}
    A property $P$ of formulas (respectively, theories) is said to be \textit{straightly definable} if there is a straight pattern $(\calC_P,\calI_P)$ such that each partitioned formula (respectively, each theory) has $P$ if, and only if, it exhibits $(\calC_P,\calI_P)$.
\end{definition}

The notion of straight definition is able to characterize all of the paradigm binary classificatory properties, like OP, IP, and TP. In his Proposition 3.3, Bailetti presents straight patterns for all of these, alongside patterns for SOP, $\TPi$ and $\TPii$ (though \cite{shelah_what_2000} states earlier that the existence of these straight patterns is immediate, without explicitly defining them). We include the straght patterns for OP and IP, both as an instructive example, and to compare with their higher arity analogues in Section \ref{higher arity straight defn section}.

\begin{proposition}\label{straight pattern examples proposition}
    Let $\vphi(x,y)$ be a partitioned formula in the language of $T$. Then
    \begin{enumerate}
        \item $\vphi(x,y)$ witnesses $\mathrm{OP}$ for $T$ if and only if $\vphi$ exhibits in $T$ the straight pattern $(\calC,\calI)$, where  
        \[\calC=\{(\{\ell\colon \ell\geq i\},\{0,\dotsm, i-1\})\colon i<\omega\}\enspace\text{and}\enspace \calI=\emptyset.\]

        \item $\vphi(x,y)$ witnesses $\mathrm{IP}$ for $T$ if and only if $\vphi$ exhibits in $T$ the straight pattern $(\calC,\calI)$, where  
        
        \[\calC=\{(X,\omega\setminus X)\colon X\subseteq \omega\}\enspace\text{and}\enspace\calI=\emptyset.\]
        
    \end{enumerate}
\end{proposition}

As noted above, TP is also straightly definable. This shows that  straight definability offers a potential improvement upon poset definability, as it is still an open question whether TP can be characterized by posets. But is there any property which is poset definable but not straightly definable? There is not: this fact is stated in passing by Garcia and Mennuni (below their Question 4.2), but it is worth articulating exactly why this is the case.

\begin{proposition}\label{poset implies straight}
    Suppose $\Sigma$ is a poset. Then there is a straight pattern $(\calC, \calI)$ such that a formula $\vphi(x,y)$ has $\Sigma$-$\mathrm{SOP}$ if, and only if, $\vphi(x,y)$ exhibits $(\calC, \calI)$. It follows that if a property is poset definable, then it is straightly definable.
\end{proposition}

\begin{proof}
    Define the pattern $(\calC,\calI)$ arising from $(\Sigma,\prec)$ as follows: 

    \[\calC=\{(\{\sigma\},\{\tau
    \})\colon \sigma\not\prec\tau\}\]
    
   \[\calI=\{(\{\sigma\},\{\tau\})\colon \sigma\prec\tau\}.\]

First, suppose $\vphi(x,y)$ has $\Sigma$-SOP witnessed by $(a_\sigma\colon \sigma\in \Sigma)$. Then the following equivalences hold for $\sigma,\tau\in\Sigma$:

\[\sigma\prec\tau\iff \models\forall x(\vphi(x,a_\sigma)\to\vphi(x,a_\tau))\iff\models \neg\exists x(\vphi(x, a_\sigma)\land\neg\vphi(x, a_\tau)).\]

The first of these equivalences comes from the definition of $\Sigma$-SOP, while the second holds by propositional logic. But the last of these statements is equivalent to the condition that $\{\vphi(x, a_\sigma),\neg\vphi(x,a_\tau)\}$ is inconsistent. Thus we can take $(a_\sigma\colon \sigma\in \Sigma)$ to be the parameters witnessing the fact that $\vphi(x,y)$ exhibits $(\calC,\calI)$, as  $\{\vphi(x, a_\sigma),\neg\vphi(x,a_\tau)\}$ is inconsistent if, and only if, $\sigma\prec\tau$. 

Conversely, if $\vphi(x,y)$ exhibits $(\calC,\calI)$, then $\sigma\prec\tau$ if and only if $\{\vphi(x, a_\sigma),\neg\vphi(x,a_\tau)\}$ is inconsistent. But again by propositional logic, this is the case exactly when $\models\forall x(\vphi(x,a_\sigma)\to\vphi(x,a_\tau))$, and thus $\vphi(x,y)$ witnesses $\Sigma$-SOP.
   
\end{proof}

It is still open whether a property being straightly definable implies that it is poset definable; TP may serve as a counterexample.

However, Garcia and Mennuni prove that an significant class of straightly definable properties are poset definable. Theorem 3.10 in \cite{garcia2022model} shows that all ``tree properties defined via maximal consistency patterns" are poset definable. The relevant notion of a ``consistency pattern" is very similar to that of an straight pattern above (specifically, it is a positive pattern with inconsistency conditions of size 2) while a   ``tree property" is the notion of exhibiting such a consistency pattern.

Until now, it was an open question whether \SOPn, for $n\geq 4$, was straightly definable. This served as one of the only known gaps in this definition. Our Theorem \ref{SOPn straightly definable} resolves this gap, showing that straight definition is effectively able to characterize all known binary patterns.

\subsection{Higher arity patterns}\label{higher arity straight defn section}

It is to our knowledge open whether $\text{IP}_k$ and $\text{FOP}_k$ are straightly definable. \cite{bailetti2024walk} proves several results relating $\text{IP}_k$ to forms of \textit{positive maximality}, i.e., the property of exhibiting all patterns in which no negations appear\footnote{Bailetti showed that theories which are $k+1$\textit{-positively maximal}, exhibiting all positive patterns where inconsistency conditions are at most size $k+1$, have $\text{IP}_k$. See his Theorem 6.11.}. We will discuss the  \textit{positively straightly definable} properties defined by these patterns in Section \ref{positive definability categoricity}. The relationship between $\text{IP}_k$ and straight patterns which allow negations is, as yet, unexplored. However, a simple and natural modification of the definitions of straight pattern and exhibition show that the use of consistency and inconsistency conditions is easily able to characterize higher arity properties.

\begin{definition}\label{k-ary straight pattern}
    For $k\geq 1$, a $k$\textit{-ary straight pattern} is a pair $(\calC, \calI)$ where each of $\calC$ and $\calI$ is a subset of $\calP(\omega^{k})\times\calP(\omega^k)\setminus\{(\emptyset,\emptyset)\}$.

    We define consistency and inconsistency conditions of a $k$-ary straight pattern as with an ordinary straight pattern.
    
\end{definition}

\begin{definition}\label{k-ary exhibits definition}
    A formula $\vphi(x,y_0,\dotsm, y_{k-1})$ \textit{exhibits} a $k$-ary straight pattern $(\calC,\calI)$ in an $L$-theory $T$ if there is a set $B=(b_{0,i},b_{1,i},\dotsm, b_{k-1,i})_{i<\omega}$ of tuples from $M\models T$, with $|b_{ji}|=|y_j|$ such that, for all $(A^+,A^-)\in \calC$ and $(Z^+,Z^-)\in \calI$, we have \[\{\vphi(x,b_{0,i_0},\dotsm, b_{k-1,i_{k-1}})\colon (i_0,\dotsm, i_{k-1})\in A^+\}\cup \{\neg\vphi(x,b_{0,j_0},\dotsm, b_{k-1,j_{k-1}})\colon (j_0,\dotsm, j_{k-1})\in A^-\}\] is consistent, and \[\{\vphi(x,b_{0,i_0},\dotsm, b_{k-1,i_{k-1}})\colon (i_0,\dotsm, i_{k-1})\in Z^+\}\cup \{\neg\vphi(x,b_{0,j_0},\dotsm, b_{k-1,j_{k-1}})\colon (j_0,\dotsm, j_{k-1})\in Z^-\}\] is inconsistent.

     A theory $T$ \textit{exhibits} a $k$-ary straight pattern $(\mathcal{C}, \mathcal{I})$ if some formula $\varphi(x, y)$ exhibits $(\mathcal{C}, \mathcal{I})$ in $T$.
\end{definition}

\begin{definition}\label{k-ary straightly definable property}
    A property $P$ of $k+1$-partitioned formulas (respectively, of theories) is said to be \textit{$k$-ary straightly definable} if there is a $k$-ary straight pattern $(\calC_P,\calI_P)$ such that each partitioned formula $\vphi(x,y_0,\dotsm, y_{k-1})$ (respectively, each theory) has $P$ if, and only if, 
    it exhibits $(\calC_P,\calI_P)$.
\end{definition}

It is clear from these definitions that the ordinary notions of a straight pattern, exhibition, and straightly definable property are just the $k=1$ case of these definitions. 

Moreover, the usual patterns witnessing the straight definability of OP and IP generalize to the $k$-ary analogues of these properties.

\begin{proposition}\label{higher arity straight pattern examples proposition}
    Let $\vphi(x,y_0,\dotsm, y_{k-1})$ be a partitioned formula in the language of $T$. Then
    \begin{enumerate}
        \item $\vphi$ witnesses $\mathrm{FOP}_k$ for $T$ if and only if $\vphi$ exhibits in $T$ the $k$-ary straight pattern $(\calC,\calI)$, where  
        \[\calC=\{(\{(i_0,\dotsm, i_{k-1})\colon i_{k-1}\leq f(i_0,\dotsm, i_{k-2})\},\{(i_0,\dotsm, i_{k-1})\colon i_{k-1}> f(i_0,\dotsm, i_{k-2})\})\colon\]
        \[ f\colon\omega^{k-1}\to\omega\enspace\text{is a function}\}\]
        
        and $\calI=\emptyset$.

        \item $\vphi(x,y)$ witnesses $\mathrm{IP}_k$ for $T$ if and only if $\vphi$ exhibits in $T$ the straight pattern $(\calC,\calI)$, where  
        
        \[\calC=\{(X,\omega^k\setminus X)\colon X\subseteq \omega^k\}\enspace\text{and}\enspace\calI=\emptyset.\]
        
    \end{enumerate}
\end{proposition}

\section{Defining $\SOP_n$ with posets and straight definitions}\label{New results on SOPn}

We show that the following poset, which naturally modifies that witnessing OP in \cite{garcia2022model}, can be embedded in some interpretable poset in each theory which witnesses \SOPn. It follows that each $\SOP_n$ property is poset definable, and thus straightly definable.

\begin{definition}
    We define the following poset $\Sigma_n$: its domain is the $n\times \omega$-indexed array $\{(i,j)\colon i<n, j<\omega\}$, and the only relations are given by \[(i,j)<(i', j') \IFF i<i' \enspace\text{and}\enspace j<j'.\] (The case $n=4$ is visualized below, where a path in this diagram represents the order relation.)
\end{definition}

\begin{center}
\begin{tikzpicture}
    
    \node(b1){$(1,0)$};
    \node(b2)[right of = b1]{$(1,1)$};
    \node(dots0)[right of =b2]{$\dotsm$};
    \node(empty)[right of =dots0]{$(1,i)$};
    \node(dots1)[right of =empty]{$\dotsm$};
    \node (ai)[below of =empty]{$(0,i)$};
    \node (adots1)[left of =ai]{$\dotsm$};
    \node (adots2)[right of =ai]{$\dotsm$};
    \node (a2)[left of =adots1]{$(0,1)$};
    \node (a1)[left of =a2]{$(0,0)$};

    \node (20)[above of =b1]{$(2,0)$};
    \node (21)[right of =20]{$(2,1)$};
    \node (2dots)[right of =21]{$\dotsm$};
    \node (2i)[right of =2dots]{$(2,i)$};
    \node (2dots2)[right of =2i]{$\dotsm $};

    \node (30)[above of =20]{$(3,0)$};
    \node (31)[right of =30]{$(3,1)$};
    \node (3dots)[right of =31]{$\dotsm$};
    \node (3i)[right of =3dots]{$(3,i)$};
    \node (3dots2)[right of =3i]{$\dotsm $};
    
    \draw[thick, ->] (a1) -- (b2);
    \draw[thick, ->](a1)--(31);
    \draw[thick, ->](b1)--(31);
    \draw[thick, ->](a2)--(3i);
    \draw[thick, ->](21)--(3i);
    \draw[thick, ->](b2)--(2i);
    \draw[thick, ->](b2)--(3i);
    \draw[thick, ->](a2)--(empty);
    \draw[thick, ->](20)--(31);

\end{tikzpicture}
\end{center}

\begin{theorem}\label{SOPn pattern theorem}
    For $n\geq 3$, a theory $T$ has $\mathrm{SOP}_n$ if, and only if, $T$ has $\Sigma_n$-$\mathrm{SOP}$.
    
\end{theorem}

\begin{proof}
    First, suppose $T$ has $\Sigma_n$-SOP, witnessed by the formula $\psi(x,y)$ and the parameters $(a_{i,j})_{i<n, j<\omega}$. 
    We claim that there is a formula that witnesses $\SOP_n$ for the parameters $(\bar{a}_i)_{i<\omega}$, where $\bar{a}_i=(a_{0,i}, a_{1,i},\dotsm, a_{n-1,i})$; i.e., the parameters in the poset's ``columns," in ascending order. 
    

    The witnessing formula will be an element of the type of the first two columns, $p(\bar{y}_0, \bar{y}_1):=\tp(\bar{a}_0, \bar{a}_1)$.

    By $\Sigma_n$-SOP, $p$ contains all formulae of the forms $\forall x(\psi(x, y_{\ell,i})\to \psi(x, y_{m,j}))$ and $\neg \forall x(\psi(x, y_{m,i})\to \psi(x, y_{\ell, j}))$, for all $\ell<m<n$, $i < j < \omega$. For $\bar{y}=(y_{0}, \ldots y_{n-1})$, $\bar{y'}=(y'_{0}, \ldots y'_{n-1})$ define 
    
    $$\pi(\bar{y}, \bar{y}')=: \bigwedge_{\ell < m < n} \forall x(\psi(x, y_{\ell})\to \psi(x, y'_{m})) \wedge \bigwedge_{\ell < m < n} \neg \forall x(\psi(x, y_{m})\to \psi(x, y'_{\ell})) $$.

    Then $\pi(\bar{a}_{i}, \bar{a}_{j})$ if $i < j$. We will show that $\pi(\bar{y}, \bar{y}')$ is a witnessing formula for $\mathrm{SOP}_{n}$. It remains to show the anti-looping clause.

    To check that we can have no loops of length $n$ in $\pi$, consider the formula 
    \[Q(\bar{y}_0,\dotsm, \bar{y}_{n-1})=: \pi(\bar{y}_0,\bar{y}_1)\wedge \pi(\bar{y}_1,\bar{y}_2)\wedge\dotsm \wedge \pi(\bar{y}_{n-2},\bar{y}_{n-1})\wedge \pi(\bar{y}_{n-1},\bar{y}_0).\] 

    Now, $Q$ implies the following chain of implications:

   $$
    \forall x(\psi(x, y_{0,1})\to \psi(x, y_{1,2})), \forall x(\psi(x, y_{1,2})\to \psi(x, y_{2,3})), \dotsm\\
    \dotsm, \forall x(\psi(x, y_{n-2,n-1})\to \psi(x, y_{n-1,0})).
  $$


    Thus, this formula implies $\forall x(\psi(x, y_{0,1})\to \psi(x, y_{n-1,0}))$. But since $0 < n-1$, we also have that $\pi(\bar{y}_{n-1},\bar{y}_0) \vdash\neg \forall x(\psi(x, y_{n-1,0})\to \psi(x, y_{0, 1}))$. Hence $Q$ is inconsistent, so $\pi(\bar{y}, \bar{y}')$ satisfies the anti-looping clause.
    

    Conversely, suppose that $\vphi(x,y)$ witnesses $\SOP_n$. By compactness, we may assume it witnesses $\SOP_n$ on parameters indexed by $(\bbQ,<)$, say $(a_q)_{q\in\bbQ}$ in $\bbM\models T$; take these parameters to have length $N$. 
Now recall that $T$ has $\Sigma$-SOP if, and only if, $\bbM\models T$ interprets a poset $\Sigma'$ into which $\Sigma$ embeds. We show that $\bbM$ interprets a poset into which $\Sigma_n$ embeds. 

Fixing distinct arbitrary parameters $c_0,\dots, c_{n-1}$ from $\bbM$, the domain for the poset consists of $n$ levels, given by the $c_i$. The domain set is $P=\bigcup_{i<n}\{c_i\}\times\bbM^N$. The poset relation is defined between pairs of variables $(z,\bar{x})$, where $\bar{x}$ has length $N$ and $z$ has length 1. It expresses the relation that holds between $(c_i, \bar{b}_0)$ and $(c_{i'},\bar{b}_1)$ when $i<i'$ and there is a $i'-i$-length path in the $\vphi$-relation starting with $\bar{b}_0$ and ending with $\bar{b}_1$. Let $\vphi^k(\bar{x}, \bar{y})$ abbreviate the following formula, which asserts that there is a length-$k$ path from $\bar{x}$ to $\bar{y}$:

$$\vphi^k(\bar{x}, \bar{y})=\exists \bar{u}_0,\dotsm, \bar{u}_{k-1}\left[\bar{x}=\bar{u}_0\land \bar{y}=\bar{u}_{k-1}\land \bigwedge_{\ell<k-1}\vphi(\bar{u}_\ell, \bar{u}_{\ell+1})\right].$$

Then we let 
    \[\psi(z,\bar{x};w,\bar{y})=(\bar{x}=\bar{y}\land z=w)\lor\bigvee_{i<i'<n}\left((z=c_i
    \land w=c_{i'})\land \vphi^{i'-i}(\bar{x},\bar{y})\right).\]

    This formula, by its construction, defines a poset relation $\preceq$ on $P$, so it suffices to show that $\Sigma_n$ embeds into $(P,\preceq)$. We claim that the function $h\colon \Sigma_n\to (P,\preceq)$ defined by $h((i,j))=(c_i, a_j)$ gives the desired embedding.

    To check that $h$ is an embedding, first suppose $(i,j)\leq(i',j')$. If $(i,j)=(i',j')$, then $h((i,j))\preceq h((i',j'))$, by the first disjunct of $\psi$. So we may assume $(i,j)<(i',j')$; hence $i<i' <n-1$ and $j<j'$. Since $\bbQ$ is dense, we may find $q_0<q_1<\dotsm <q_{i'-i-1}$ with $j=q_0$ and $j'=q_{i'-i-1}$. We have $\vphi(a_{q_0}, a_{q_1}), \vphi(a_{q_1},a_{q_2}),\dotsm, \vphi(a_{q_{i-i'-2}},a_{q_{i'-i-1}})$, and thus $\models \psi(c_i,a_j;c_{i'},a_{j'})$. Thus we have that $h((i,j))\prec h((i',j'))$; so in either case, $h((i,j))\preceq h((i',j'))$.
    
    Now, suppose that $(i,j)\not\leq(i',j')$. Then at least one of $i\geq i'$ and $j\geq j'$ is true. If $i>i'$, then there is no disjunct of $\psi$ that could make $\psi(c_i,a_j;c_{i'},a_{j'})$ true, by the equality conditions on the $c_k$ parameters. If $i=i'$ we must look to the first conjunct of $\psi$, and since in this case $j\neq j'$, we have $\not\models \psi(c_i,a_j;c_{i'},a_{j'})$. Finally, suppose $i<i'$ and $j\geq j'$. If we had $\models \psi(c_i,a_j;c_{i'},a_{j'})$, then it would follow that there is a $i'-i$-length $\vphi$-path from $a_{j}$ to $a_{j'}$. The fact that $\vphi$ witnesses $\SOP_n$ on $(a_q)_{q\in\bbQ}$ implies, as in the previous paragraph, that there is an $n-(i'-i)$-length $\varphi$-path from $a_{j'}$ to $a_{j}$. But together with the links in the $i'-i$-length $\vphi$ path from $a_{j}$ to $a_{j'}$, this would show that there is an $n$-length $\vphi$-loop, contradicting $\SOP_n$'s condition that such loops are inconsistent. Thus in any case we must have $h((i,j))\not\preceq h((i'j'))$. We conclude that $h$ is an embedding of posets, and thus that $T$ has $\Sigma_n$-SOP.
    
\end{proof}

We have now proved that the $\SOP_n$ properties are poset definable. It immediately follows by Proposition \ref{poset implies straight} that the $\SOP_n$s are straightly definable. 

\begin{theorem}\label{SOPn straightly definable}
    For $n\geq 3$, $\mathrm{SOP}_n$ is straightly definable.
\end{theorem}

\begin{proof}
    By Theorem \ref{SOPn pattern theorem}, each $\SOP_n$ is poset definable. Thus by Proposition \ref{poset implies straight}, $\SOP_n$ is straightly definable.
\end{proof}

This answers the open question mentioned in \cite{bailetti2024walk} (see p. 23) about the properties $\text{SOP}_n$ for $n\geq 4$; as noted there, \cite{kaplan2024generic} provide a positive pattern for $\SOP_3$ (though Fact 1.3 of \cite{shelah_usvyatsov_2008} also straightforwardly implies positive straight definability of $\mathrm{SOP}_{3}$).  While Shelah claimed to have demonstrated that all of the properties $\SOP_n$  are straightly definable in \cite{shelah_what_2000} (see Observation 5.19), his proof has long been known to be erroneous; we discuss why in detail in Appendix \ref{shelah ``straight definition"}. By providing a straight definition for all of these properties, we have filled a significant gap in the theory of model-theoretic patterning notions. The following table represents what we now know, for all of Shelah's classical model-theoretic patterning properties. Note that the property $\TPi$ is equivalent to $\SOP_2$, so that property is included in the chart. The notion of a \textit{positively straightly definable property} is defined in the following section, Definition \ref{positively straightly definable property}.

\:

\resizebox{\textwidth}{!}{
  \begin{tabular}{l|rrrrrrrrr}
     & $\mathrm{OP}$ & $\mathrm{IP}$ & $\mathrm{TP}$ & $\mathrm{SOP}_{1}$ & $\mathrm{SOP}_{2}$  & $\mathrm{SOP}_{3}$ & $\mathrm{SOP}_{\geq 4}$ & $\mathrm{SOP}$ & $\mathrm{TP}_{2}$ \\
     \hline
     trace definable\footnotemark  & Y & Y & N & N & N & N & N & N & N \\
     positively straightly definable  & Y\footnotemark & Y\footnotemark[\value{footnote}] & Y & Y & Y & Y\footnotemark & N\footnotemark & N\footnotemark[\value{footnote}] & Y \\
     poset definable  & Y\footnotemark & Y\footnotemark[\value{footnote}] & ? & Y\footnotemark[\value{footnote}] & Y\footnotemark[\value{footnote}]\textsuperscript{,}\footnotemark & Y\footnotemark[\numexpr\value{footnote}-1\relax] & \textbf{Y} & Y & Y\footnotemark[\numexpr\value{footnote}-1\relax] \\
     straightly definable  & Y\footnotemark[\numexpr\value{footnote}-2\relax] & Y\footnotemark[\numexpr\value{footnote}-2\relax] & Y & Y & Y & Y\footnotemark[\numexpr\value{footnote}-3\relax] & \textbf{Y} & Y\footnotemark[\numexpr\value{footnote}-2\relax] & Y \\
  \end{tabular}
}

\addtocounter{footnote}{-6} 
\stepcounter{footnote}\footnotetext{ All memberships and non-memberships in this row due to \cite{walsberg2025trace}}
\stepcounter{footnote}\footnotetext{\cite{day2025results}}
\stepcounter{footnote}\footnotetext{\cite{shelah_usvyatsov_2008}}
\stepcounter{footnote}\footnotetext{\cite{bailetti2024walk}}
\stepcounter{footnote}\footnotetext{\cite{garcia2022model}}
\stepcounter{footnote}\footnotetext{Uses $\mathrm{NSOP}_{1} = \mathrm{NSOP}_{2}$.}

\:

\:

We'll return to the one remaining question mark, that of the poset definability of TP, in Section \ref{open questions}.

\section{Positive straight definability in countably categorical theories}\label{positive definability categoricity}

By the above Theorem \ref{SOPn straightly definable}, we now know that all of the properties within the original classification-theoretic hierarchy defined by Shelah fall within the notion of a model-theoretic patterning property given by straight definability. These include the order property $\mathrm{OP}$, the independence property $\mathrm{IP}$, the tree property $\mathrm{TP}$, the strict order property $\mathrm{SOP}$, $\mathrm{TP}_{2}$, $\mathrm{SOP}_{n}$ for integers $1 \leq n \leq 3$, and now, completing the picture for straight definability, $\mathrm{SOP}_{n}$ for integers $n \geq 4$. However, much remains unknown on a family of central and longstanding questions about this original hierarchy: whether some of the properties within this hierarchy (or intersections of these properties) are equal or not. As we show in this section, viewing properties alike as \textit{positively} straightly definable gives general techniques for analyzing these problems and giving quantifier complexity bounds on formulas witnessing implication between them.

The most iconic of these equivalence questions involve the properties $\mathrm{SOP}_{n}$ for $1 \leq n \leq 3$. These properties were introduced in \cite{Dzamonja2004}, and first attested by \cite{shelah_what_2000} in notes, citing this work then in preparation with Džamonja, based on lectures delivered at Rutgers University in fall of 1997. Compared to $\mathrm{SOP}_{n}$ for integers $n \geq 3$, $\mathrm{SOP}_{1}$ and $\mathrm{SOP}_{2}$ are defined differently:

\begin{definition}
\label{1-strict order property}

A theory $T$ has $\mathrm{SOP}_{1}$ if there exists a formula $\varphi(x, y)$ and tuples $\{b_{\eta}\}_{\eta \in 2^{<\omega}}$ (from some model $M \models T$) so that $\{\varphi(x, b_{\sigma \upharpoonright n})\}_{n < \omega}$ is consistent for any $\sigma \in 2^{\omega}$, but for any $\eta_{2} \unrhd \eta_{1} \smallfrown \langle 0\rangle$, $\{\varphi(x, b_{\eta_{2}}), \varphi(x, b_{\eta_{1} \smallfrown \langle 1\rangle})\}$ is inconsistent. Otherwise it is $\mathrm{NSOP}_{1}$.
\end{definition}

\begin{definition}
\label{2-strict order property}

A theory $T$ has $\mathrm{SOP}_{2}$ if there exists a formula $\varphi(x, y)$ and tuples $\{b_{\eta}\}_{\eta \in 2^{<\omega}}$ so that $\{\varphi(x, b_{\sigma \upharpoonright n})\}_{n < \omega}$ is consistent for any $\sigma \in 2^{\omega}$, but for incomparable $\eta_{1}$ and $\eta_{2}$, $\{\varphi(x, b_{\eta_{1}}), \varphi(x, b_{\eta_{2}})\}$ is inconsistent. Otherwise it is $\mathrm{NSOP}_{2}$.
\end{definition}

Notice that these properties are straightly definable (Definition \ref{straightly definable property}), but in particular, they are straightly definable using only consistency and inconsistency among the positive instances of $\varphi(x, y)$, with no negative instances of this formula, or instances of $\neg\varphi(x, y)$, being used in the definition. They therefore satisfy a definition of being a patterning property defined by \cite{bailetti2024walk} which, again slightly modifying the terminology there, we will call \textit{positive straight definability}. We can think of positive straight definability as another notion of patterning property, alongside the notions of poset definability, trace definability and straight definability exposited in Section 3 above.

\begin{definition}\label{positively straightly definable property}
    
    (1) A straight pattern $(\mathcal{C}, \mathcal{I})$ is a \textit{positive pattern} if $X^{-} = \emptyset$ for all $(X^{+}, X^{-}) \in \mathcal{C}$, and $Y^{-} = \emptyset$ for all $(Y^{+}, Y^{-}) \in \mathcal{I}$.

    (2) A property $P$ of formulas (respectively, theories) is said to be \textit{positively straightly definable} if there is a positive pattern $(\calC_P,\calI_P)$ such that each partitioned formula $\vphi(x,y)$ (resp. theory $T$) has $P$ if, and only if, $\vphi(x,y)$ (resp. $T$) exhibits $(\calC_P,\calI_P)$.
\end{definition}

In addition to $\mathrm{SOP}_{1}$ and $\mathrm{SOP}_{2}$, the tree property $\mathrm{TP}$ (Definition \ref{tree property} above) is clearly positively straightly definable, as is $\mathrm{TP}_{2}$, which we will define below (Remark \ref{preservation of sopn}). We also noted above that \cite{shelah_usvyatsov_2008} showed that $\mathrm{SOP}_{3}$ is positively straightly definable.\footnote{More precisely, they provide a two-formula characterization that closely resembles this; this one-formula characterization also resembles the ``compatible order property" of \cite{Mal10}.} Specifically:

\begin{fact}
\label{positive straight definition of sop3}

A theory has $\mathrm{SOP}_{3}$ if and only if there is an array $\{a_{i}, b_{i}\}_{i < \omega}$ and a formulas $\varphi(x, y)$ with

    (1) For $m < n$, $\{\varphi(x, b_{i})\}_{i \leq m} \cup \{\varphi(x, a_{i})\}_{m < i \leq n}$ is consistent.

    (2) For $ i < j$, $\{\varphi(x, a_{i}), \varphi(x, b_{j})\}$ is inconsistent.

\end{fact}

Moreover \cite{day2025results} showed that instability ($\mathrm{OP}$), as well as the independence property ($\mathrm{IP}$), well-known to be straightly definable, are in fact positively straightly definable. However, $\mathrm{SOP}$, as well as $\mathrm{SOP}_{n}$ for $n \geq 4$, have been shown by \cite{bailetti2024walk} (Example 5.7 there) not to be positively straightly definable.\footnote{It is thus unsurprising that the straight definitions for $\SOP_n$ arising from the poset definition we gave in Section \ref{New results on SOPn} involve negated instances of $\vphi$; otherwise, they would violate Bailetti's result.}

A contemporary development that has transformed the prevailing understanding of classification theory has been the discovery that there is at least one especially powerful equivalence between model-theoretic patterning properties--in particular, between two of the main positively straightly definable properties. There is precedent for implications between model-theoretic patterning properties: \cite{shelah1990classification} showed that an unstable theory either has the independence property or the strict order property, and also showed that a non-simple theory either has $\mathrm{TP}_{1}$ (equivalent to $\mathrm{SOP}_{2}$) or $\mathrm{TP}_{2}$. However, these are combinatorial results about formulas; i.e., they hold at the level of formulas up to Boolean combination (so for example, if a formula is unstable, it has a Boolean combination that either exhibits the independence property or the strict order property). By contrast, the following equivalence, which is proven by developing a theory of independence for $\mathrm{NSOP}_{2}$ theories, is not true at the level of formulas up to Boolean combination (as shown by \cite{AhnKim2024natp}, in Section 6 there). Answering a longstanding open question posed by \cite{Dzamonja2004} and mentioned earlier in \cite{shelah_what_2000} (also citing this work with Džamonja while it was in preparation), \cite{Mutchnik2026nsop2} shows:

\begin{fact}\label{sop1 equals sop2}
 Every theory with $\mathrm{SOP}_{1}$ has $\mathrm{SOP}_{2}$.
\end{fact}

However, the following question, also posed by \cite{Dzamonja2004} and mentioned in \cite{shelah_what_2000}, remains one of the main open problems of model theory:

\begin{question}\label{is sop2 equal to sop3}
    Does every theory with $\mathrm{SOP}_{2}$ have $\mathrm{SOP}_{3}$?
\end{question}

So implications between model-theoretic patterning properties are of central interest, particularly within the setting of positively straightly definable properties. The results of this section will be about implications or equivalences between model-theoretic patterning properties in the setting of \textit{countably categorical theories}.

Fact \ref{sop1 equals sop2} says that, if a given theory has a formula $\varphi(x, y)$ with $\mathrm{SOP}_{1}$, it has is a formula $\varphi'(x, y)$ that exhibits $\mathrm{SOP}_{2}$. In a problem session at the 2023 meeting on neostability theory at BIRS, Byunghan Kim asked the informal question of whether the $\mathrm{SOP}_{2}$ formula $\varphi'(x, y)$ can be described in terms of the $\mathrm{SOP}_{1}$ formula $\varphi(x, y)$ (see \cite{BIRS23w5145} for a report on the BIRS conference where Byunghan Kim asked this question). In fact, a close analysis of the second author's proof that $\mathrm{SOP}_{1} = \mathrm{SOP}_{2}$ will demonstrate that, if $\varphi(x, y)$ exhibits $\mathrm{SOP}_{1}$ in a theory, there is a formula $\varphi'(x', y')$ exhibiting $\mathrm{SOP}_{2}$  which is a Boolean combination of existential and universal formulas in $\varphi'(x, y)$. Particularly, if a theory has a formula $\varphi(x', y')$ exhibiting $\mathrm{SOP}_{1}$, then it has a formula $\varphi'(x', y')$ exhibiting $\mathrm{SOP}_{2}$ which is a $\exists\forall$-formula in $\varphi(x, y)$. By a $\exists\forall$-formula in $\varphi(x, y)$, we mean a formula of the form $\exists u \forall v \psi(u, v, x', y')$ for tuples of variables $u, v, x', y'$, where $\psi(u, v, x', y')$ is a Boolean combination of instances of $\varphi(x, y)$.\footnote{Equivalently, $\varphi'(x', y')$, within the expansion by a definable predicate for $\varphi(x, y)$, is equivalent to a $\forall\exists$-formula in the reduct to the language consisting of that predicate.} We give more details about quantifier complexity in the proof that $\mathrm{SOP}_{1}$ is equal to $\mathrm{SOP}_{2}$ in Appendix \ref{formula complexity sop1 sop2}.

The main result of this section will be to show that a similar phenomenon applies to all equivalences or implications between positively straightly definable properties of theories, \textit{assuming countable categoricity}. In exchange for this requirement, we only need to assume that the implication holds at the level of countably categorical theories, though we are often interested in general implications, such as in Question \ref{is sop2 equal to sop3} above on $\mathrm{SOP}_{2}$ and $\mathrm{SOP}_{3}$.

\begin{theorem} \label{theorem on implications between positively straightly definable properties in countably categorical theories}
    Let $P$ and $Q$ be positively straightly definable properties, and suppose that for every countably categorical theory $T$,

    $$T \mathrm{\: has \:} P \Rightarrow T \mathrm{\: has \:} Q.$$

    Then for any countably categorical theory $T$, for every formula $\varphi(x, y)$ that exhibits $P$ in $T$, there is a formula $\varphi'(x', y')$ exhibiting $Q$ in $T$ which is a $\exists\forall$-formula in $\varphi(x, y)$.

\end{theorem}

Our proof uses results of \cite{bodirsky2025taking} and \cite{saracino1973model} on model companions of complete theories. The reader may be more familiar with model companions within the context of constructing concrete examples of first-order theories as model companions of a particular universal theory (possibly a universal expansion of a Morleyized base theory). For example, the theory of the random graph is constructed as the model companion of the theory of undirected graphs, a universal theory. However, throughout this paper we are implicitly assuming that a theory is complete, as is common when discussing first-order theories in general, especially within the context of model-theoretic patterning properties. Therefore, it may be useful to review the definition of the model companion of a complete theory. Note that this really is a nontrivial definition, because the theory $T$ in this definition is \textit{not} assumed to be Morleyized. 

\begin{definition}
\label{model companion of a complete theory}

Let $T$ be a (complete) theory; then a (complete) theory $T'$ is the \textit{model companion} of $T$ if:

\begin{itemize}
    \item every model of $T$ embeds as an induced substructure into a model of $T'$
    \item every model of $T'$ embeds as an induced substructure into a model of $T$
    \item the theory $T'$ is model complete (i.e., if $M \models T'$, $M' \models T'$ and $M \subset M'$, then $M \prec M'$).
\end{itemize}

\end{definition}

We now discuss preservation of positively straightly definable properties under model companions. \cite{bodirsky2025taking} claim to show that, for $\mathcal{P}$ any straightly definable property, if $T$ has a model companion $T'$, then if $T'$ has the property $\mathcal{P}$, $T$ also has the property $\mathcal{P}$. However, there appears to be an error in the proof of this result for straightly definable properties, as opposed to positively straightly definable properties.\footnote{The error is the claim within the proof of Theorem 3.22 of their paper that, for $\mathcal{P}^{+}$ the \textit{positivization} of a straightly definable property $\mathcal{P}$,  $\mathcal{P}^{+}$ implies $\mathcal{P}$ for the theory $T$. We give an outline of the translation of their arguments (which they give first for \textit{core companions} and then infer for model companions) into a direct argument for their claimed result on model companions. Let $(\mathcal{C}, \mathcal{I})$ be a straight pattern for $\mathcal{P}$. Define an \textit{obstruction} of a formula $\varphi(x)$ in a theory $T$ to be an \textit{existential} formula $\psi(x)$ such that $T \models \forall x(\psi(x) \to \neg\varphi(x))$. Then say a theory has the property $\mathcal{P}^{+}$ if there is a formula $\varphi(x, y)$, a set $B=(b_i)_{i<\omega}$ of $|y|$-length tuples from $M\models T$, and obstructions $\psi(x, y)$ of $\varphi(x, y)$, and $\theta_{(q,r)}(y_1, \dots, y_q, y_{q+1}, \dots, y_{q+r})$ of $\exists x \left( \bigwedge_{i=1}^q \phi(x; y_i) \wedge \bigwedge_{i=1}^r \psi(x; y_{q+i}) \right)$ for each $q, r < \omega$, such that

for all $(A^+,A^-)\in \calC$ and $(Z^+,Z^-)\in \calI$, we have 
    \[\{\vphi(x,b_i)\colon i\in A^+\}\cup \{\psi(x,b_j)\colon j\in A^-\}\] is consistent, and

    $$(b_{i})_{i \in Z^{+}}(b_{i})_{i \in Z^{-}} \models \theta_{(|Z^{+}|, |Z^{-}|)}(y_1, \dots, y_{|Z^{+}|}, y_{|Z^{+}|+1}, \dots, y_{|Z^{+}|+|Z^{-}|}).$$

    Now suppose a theory $T$ has a model companion $T^{+}$, and that $T^{+}$ has $\mathcal{P}$. It is well known that every formula in a model complete theory is equivalent to an existential formula. So there is an existential formula $\varphi(x, y)$ exhibiting $\mathcal{P}^{+}$ in $T'$, say with the obstruction $\psi(x, y)$ to $\varphi(x, y)$. Then we can show that $\varphi(x, y)$ exhibits $\mathcal{P}^{+}$ in $T$, also with the obstruction $\psi(x, y)$ to $\varphi(x, y)$. It is then claimed that the assertion that $T$ has $\mathcal{P}^{+}$ implies that $T$ has $\mathcal{P}$, as desired.

    However, if $\mathcal{P}$ is only assumed to be straightly definable and not positively straightly definable, this last claim is not justified. Specifically, while it is known that $T \models \forall x  y(\psi(x, y) \to \neg\varphi(x, y))$, it is not known that $T \models \forall x y(\psi(x, y) \leftrightarrow \neg\varphi(x, y))$, even if $\psi(x, y)$ is chosen so that $T' \models \forall x y(\psi(x, y) \leftrightarrow \neg\varphi(x, y))$. So in $T$, while the second condition of $\mathcal{P}^{+}$ does imply that $\{\vphi(x,b_i)\colon i\in Z^+\}\cup \{\psi(x,b_j)\colon j\in Z^-\}$ is inconsistent for $(Z^{+}, Z^{-}) \in \mathcal{I}$, that does not imply that $\{\vphi(x,b_i)\colon i\in Z^+\}\cup \{\neg\vphi(x,b_j)\colon j\in Z^-\}$ is inconsistent for $(Z^{+}, Z^{-}) \in \mathcal{I}$. If $\mathcal{P}$ is assumed to be \textit{positively} straightly definable, though, the claim that $\mathcal{P}^{+}$ implies $\mathcal{P}$ is justified: the second condition of $\mathcal{P}^{+}$ does imply that $\{\vphi(x,b_i)\colon i\in Z^+\}$ is inconsistent for $(Z^{+}, \emptyset) \in \mathcal{I}$.
    
    For a clarification of their proof, an upcoming preprint of the second author will contain an exposition of the result of \cite{bodirsky2025taking} on preservation of classification-theoretic properties under model companions. While that exposition will assume $|Z^{+}|=2$ for each $(Z^{+}, \emptyset) \in \mathcal{I}$, that will only be for reasons of notational simplicity, and the exposition will generalize to any positively straightly definable property.}

    This justifies why the results of this section, particularly Theorem \ref{theorem on implications between positively straightly definable properties in countably categorical theories}, are about \textit{positively} straightly definable properties, rather than straightly definable properties in general. (At the conclusion of this section, we will give some remarks applying the proof of Theorem \ref{theorem on implications between positively straightly definable properties in countably categorical theories} outside of the positively straightly definable context, but these will not be general results on straightly definable properties.)

Not only do \cite{bodirsky2025taking} give a correct proof of the fact that model companions preserve positively straightly definable properties; their proof even shows that, for the existential formula exhibiting a positively straightly definable property in the model companion of a theory, \textit{the same existential formula} exhibits that positively straightly definable property in the original theory. Note in the following that if a model complete theory has a positively straightly definable property $P$, some existential formula exhibits $P$, because every formula in a model complete theory is equivalent to an existential formula.

\begin{fact}\label{preservation under model companions}
(\cite{bodirsky2025taking}, proof of Theorem 3.22)

Let a theory $T$ have a model companion $T'$, and let $P$ be a positively straightly definable property. Suppose that $T'$ has $P$, exhibited by the \emph{existential} formula $\varphi(x, y)$. Then $\varphi(x, y)$ also exhibits $P$ in $T$.
\end{fact}

The place in the proof of Theorem \ref{theorem on implications between positively straightly definable properties in countably categorical theories} where we will use countable categoricity will be in applying the following fact, due to \cite{saracino1973model}:

\begin{fact}\label{countably categorical theories have model companions}

Every countably categorical theory has a countably categorical model companion.

\end{fact}

We are now ready to prove the main theorem of this section, which will be a consequence of the two previous facts.

\begin{proof}(of Theorem \ref{theorem on implications between positively straightly definable properties in countably categorical theories}) Let $P$ be a positively straightly definable property and $T$ a countably categorical theory in which $\varphi(x, y)$ exhibits $P$, and let $Q$ be another positively straightly definable property such that $P$ implies $Q$ for countably categorical theories. We show that there is a formula $\varphi'(x', y')$ which is a $\exists\forall$-formula in $\varphi(x, y)$ and which exhibits $Q$ in $T$. Let $R(x,y)$ be a definable predicate for $\varphi(x,y)$, and for $n < \omega$, let $I_{n}(y_{1}, \ldots, y_{n})$ be a definable predicate for $\neg \exists x \bigwedge^{n}_{i=1} \varphi(x, y_{n})$. Let $T_{0}$ be the reduct of the expansion of $T$ by these definable predicates to the language consisting of just the predicates $R$ and $I_{n}$ for $n < \omega$. Then $T_{0}$ will be countably categorical, so by Fact \ref{countably categorical theories have model companions} it will have a countably categorical model companion $T'$. We show the following claim:

\begin{claim}
The model companion $T'$ of $T_{0}$ has $P$.
\end{claim}

\begin{proof}(of claim)
We will show in particular that $R(x, y)$ exhibits $P$ in $T'$. The proof of this claim will formally resemble some arguments in \cite{bodirsky2025taking} for Fact \ref{preservation under model companions}.

First, by definition of $R$ and the $I_{n}$, for $n \leq \omega$ there are no $a, b_{1}, \ldots, b_{n}$ in any model of $M\models T_{0}$ such that $M \models \bigwedge^{n}_{i=1} \varphi(a, b_{i})$. So because every model of $T'$ embeds into a model of $T_{0}$, there are also no $a, b_{1}, \ldots, b_{n}$ in any model $M' \models T'$ such that $M' \models \bigwedge^{n}_{i=1} \varphi(a, b_{i})$. 

Second, by definition of $R$ and the $I_{n}$, and the assumption that $\varphi(x, y)$ exhibits $P$ in $T$, for $(\mathcal{C}, \mathcal{I})$ the straight pattern defining $P$, there are $\{b_{i}\}_{i < \omega}$ in a model $M \models T_{0}$ such that, for $(A^{+}, \emptyset) \in \mathcal{C}$, $\{R(x, b_{i})\}_{i \in A^{+}}$ is consistent, while, for $(Z^{+}, \emptyset) \in \mathcal{I}$, $M \models I_{n}(b_{i_{1}}, \ldots b_{i_{n}})$ for $i_{1}, \ldots i_{n}$ an enumeration of $Z^{+}$. It follows that there are $\{b_{i}\}_{i < \omega}$, $\{a_{A^{+}}\}_{(A^{+}, \emptyset) \in \mathcal{C}}$ in a model $M \models T_{0}$ such that, for $(A^{+}, \emptyset) \in \mathcal{C}$, $a_{A^{+}} \models \{R(x, b_{i})\}_{i \in A^{+}}$, while, for $(Z^{+}, 0) \in \mathcal{I}$, $M \models I_{n}(b_{i_{1}}, \ldots b_{i_{n}})$ for $i_{1}, \ldots i_{n}$ an enumeration of $Z^{+}$. Because every model of $T_{0}$ embeds into a model of $T'$, there are $\{b_{i}\}_{i < \omega}$, $\{a_{A^{+}}\}_{(A^{+}, \emptyset) \in \mathcal{C}}$ in a model $M' \models T'$ such that, for $(A^{+}, \emptyset) \in \mathcal{C}$, $a_{A^{+}} \models \{R(x, b_{i})\}_{i \in A^{+}}$, while, for $(Z^{+}, 0) \in \mathcal{I}$, $M' \models I_{n}(b_{i_{1}}, \ldots b_{i_{n}})$ for $i_{1}, \ldots i_{n}$ an enumeration of $Z^{+}$. Then for $(A^{+}, \emptyset) \in \mathcal{C}$, $\{R(x, b_{i})\}_{i \in A^{+}}$ will be consistent, realized by $a_{A^{+}}$. However, for $(Z^{+}, \emptyset) \in \mathcal{I}$, $\{R(x, b_{i})\}_{i \in Z^{+}}$ will be inconsistent, by the previous paragraph. So $R(x, y)$ will in fact exhbit $P$ in $T'$.

\end{proof}

So by the claim, and by assumption, $T'$ has $Q$. As dicussed before the statement of Fact \ref{preservation under model companions}, because $T'$ is model complete, so every formula is equivalent to an existential formula in $T$, some existential formula $\varphi'(x', y')$ exhibits $Q$ in $T'$. Then by Fact \ref{preservation under model companions}, $\varphi'(x', y')$ exhibits $Q$ in $T_{0}$. So $\varphi'(x', y')$ exhibits $Q$ in the expansion of $T$ by the definable predicates $R$ and $I_{n}$. But $\varphi'(x', y')$ is an existential formula in the language with predicates for $R$ and the $I_{n}$. And $R$ and the $I_{n}$  both have definitions in $T$ equivalent to formulas of the form $\forall \overline{u} \psi(\overline{u}, \overline{v})$, where $\psi(\overline{u}, \overline{v})$ is a Boolean combination of instances of $\varphi(x, y)$. So $\varphi'(x', y')$ is equivalent to a $\exists\forall$-formula in $\varphi(x, y)$. And because $\varphi'(x', y')$  exhibits $Q$ in the expansion of $T$, this $\exists\forall$-formula in $\varphi(x, y)$ will exhibit $Q$ in $T$, as desired.

\end{proof}

Having proven Theorem \ref{theorem on implications between positively straightly definable properties in countably categorical theories}, we next give two applications of this general phenomenon within positively straightly definable properties to open problems in classification theory. We start with the question of whether every theory with $\mathrm{SOP}_{2}$  has $\mathrm{SOP}_{3}$, Question \ref{is sop2 equal to sop3} above. We see that, if the answer to this open question is yes, then in the case of countably categorical theories this answer is exhibited at the level of $\exists\forall$-formulas. Specifcally, as a corollary of Theorem \ref{theorem on implications between positively straightly definable properties in countably categorical theories}, we have the following. (Here, when we say a formula ``exhibits $\mathrm{SOP}_{3}$", let us say that this means exhibiting the standard \textit{positive straight definition} of $\mathrm{SOP}_{3}$ from Fact \ref{positive straight definition of sop3} above. Stated in this form, the corollary follows directly from the theorem. However, by Remark \ref{preservation of sopn} below that the properties $\mathrm{SOP}_{n}$, according to their original definitions in terms of $n$-cycles, are preserved under model companions by \cite{bodirsky2025taking}, there would also be no harm in defining ``exhibit $\mathrm{SOP}_{3}$" in terms of the original definition in terms of $3$-cycles (Definition \ref{definition of sopn}).)

\begin{corollary}
Suppose that every theory with $\mathrm{SOP}_{2}$  has $\mathrm{SOP}_{3}$. Then if $T$ is countably categorical, and $\varphi(x, y)$ exhibits $\mathrm{SOP}_{2}$ in $T$, there is some $\exists\forall$-formula $\varphi'(x', y')$ in $\varphi(x, y)$ exhibiting $\mathrm{SOP}_{3}$ in $T$.
\end{corollary}

Our next application is to a problem which is specific to the countably categorical setting. This problem involves the property of being \textit{low}, which is also defined in terms of model-theoretic patterning properties:

\begin{definition}
A theory is \textit{low} if there is no formula $\varphi(x, y)$, strictly increasing sequence $k_{0} < \ldots < k_{n} < \ldots < \omega$ and $\{b^{n}_{i}\}_{i, n < \omega}$ such that for each $n< \omega$ and $i_{1}, \ldots, i_{k_{n}} < \omega$, $\{\varphi(x, b^{n}_{i_{j}})\}^{k_{n}}_{j =1}$ is consistent, but for each $n< \omega$ and distinct $i_{1}, \ldots, i_{k_{n}+1} < \omega$, $\{\varphi(x, b^{n}_{i_{j}})\}^{k_{n}+1}_{j =1}$ is inconsistent.
\end{definition}

So, reindexing, being nonlow is equivalent to a disjunction $\bigvee P_{i}$ of positively straightly definable properties of theories $P_{i}$, where the $P_{i}$ correspond to the strictly increasing sequences $k_{0} < \ldots < k_{n} < \ldots < \omega$. An open question posed by \cite{palacin2012omega} is:

\begin{question}
Is every countably categorical simple theory low?\footnote{It is not even known whether every countably categorical $\mathrm{NSOP}_{1}$ theory is low, or whether every countably categorical $\mathrm{NTP}_{2}$ theory is low. Of course, an analogous statement to the below Corollary \ref{ea-tree for countably categorical nonlow} holds in either case, because both of these properties are positively straightly definable. }
\end{question}

If the answer is yes, then each of the positively straightly definable properties $P_{i}$ implies the tree property in the countably categorical setting. So we see as a corollary of Theorem \ref{theorem on implications between positively straightly definable properties in countably categorical theories} that if the answer to this question is yes, this answer is exhibited at the level of $\exists\forall$-formulas.

\begin{corollary}\label{ea-tree for countably categorical nonlow}
Suppose that every countably categorical simple theory is low. Then if $T$ is countably categorical, and $\varphi(x, y)$ exhibits nonlowness in $T$, there is some $\exists\forall$-formula $\varphi'(x', y')$ in $\varphi(x, y)$ exhibiting the tree property in $T$.    
\end{corollary}

Notice that countable categoricity is not even an extra assumption here, because the question already only asks if every \textit{countably categorical} simple theory is low; it is well-known that not every simple theory is low in general. (See \cite{CasanovasKim1998} for one example.)

We conclude this section by giving some additional applications of the technique used to prove Theorem \ref{theorem on implications between positively straightly definable properties in countably categorical theories}, though these no longer necessarily involve positively straightly definable properties.

\begin{remark}\label{preservation of sopn}

While the arguments of \cite{bodirsky2025taking} do not, as they claim, apply to all straightly definable properties, they do show (Proposition 3.24 of that paper) that $\mathrm{SOP}_{n}$ is preserved under model companions for integers $n \geq 3$. (Because it was previously not known that $\mathrm{SOP}_{n}$ is straightly definable for integers $n \geq 4$, these properties $\mathrm{SOP}_{n}$ are handled in their paper as a separate case, despite their claims about general straightly definable properties.) In fact, they show as in Fact \ref{preservation under model companions} that if $T$ has a model companion $T'$, and $R(x,y)$ is an existential formula exhibiting $\mathrm{SOP}_{n}$ in $T'$ (in the sense of Definition \ref{definition of sopn} in terms of $n$-cycles), then $R(x, y)$ exhibits $\mathrm{SOP}_{n}$ in $T$.

An open problem posed in \cite{chernikov2014theories} asks whether, for $n \geq 3$, every theory with $\mathrm{SOP}_{n}$ has either $\mathrm{SOP}_{n+1}$ or $\mathrm{TP}_{2}$.\footnote{As an edge case of this problem, it is also open whether every theory with the tree property has either $\mathrm{TP}_{2}$ or $\mathrm{SOP}_{3}$, and an analogous version of the conclusion of this remark applies to this question.} A theory is defined to have $\mathrm{TP}_{2}$ if there is a formula $\varphi(x, y)$ and $\{b_{ij}\}_{i, j < \omega}$ such that for all $i< \omega$ and distinct $j_{1}, j_{2}< \omega$, $\{\varphi(x, b_{ij_{1}}), \varphi(x, b_{ij_{2}})\}$ is inconsistent, but for $\sigma: \omega \to \omega$, $\{\varphi(x, b_{i\sigma(i)})\}_{i < \omega}$ is consistent. So $\mathrm{TP}_{2}$ is another example of a positively straightly definable property.

Since the results of \cite{bodirsky2025taking} on preservation under model companions apply to $\mathrm{SOP}_{n}$ for $n \geq 3$ just as well as to positively straightly definable properties, we may reason about this open problem just as in the proof of Theorem \ref{theorem on implications between positively straightly definable properties in countably categorical theories}. Specifically, suppose it is in fact the case that every theory with $\mathrm{SOP}_{n}$ has either $\mathrm{SOP}_{n+1}$ or $\mathrm{TP}_{2}$. Then if $T$ is countably categorical, and $\varphi(x, y)$ exhibits $\mathrm{SOP}_{n}$ in $T$, there is some $\exists\forall$-formula $\varphi'(x', y')$ in $\varphi(x, y)$ exhibiting either $\mathrm{SOP}_{n+1}$ or $\mathrm{TP}_{2}$ in $T$.

\end{remark}

We give one last unconditional application of the technique of Theorem \ref{theorem on implications between positively straightly definable properties in countably categorical theories}. This will involve the property of \textit{supersimplicity}, which is given by a pattern in multiple formulas, rather than just one.\footnote{While we have been almost exclusively focused on local properties, where some pattern is defined by a single formula, supersimplicity is just one example of a model-theoretic property which is not local, but still has a patterning ``flavor". Other examples include various notions of rank in first-order theories. It would be interesting to see future work which attempts to characterize all local and multi-formula properties alike using some definitional scheme.} A theory is \textit{supersimple} if there is no family of formulas $\{\varphi^{i}(x, y^{i})\}_{i < \omega}$ with a family of positive integers $\{k_{i}\}_{i < \omega}$ and parameters $\{b_{\eta}\}_{\eta \in \omega^{< \omega}}$ such that, for $\sigma \in \omega^{\omega}$, $\{\varphi^{n}(x, b_{\sigma|_{n}})\}_{n < \omega}$ is consistent, but for $\eta \in \omega^{< \omega}$, any $k_{|\eta|}$ many distinct formulas in $\{\varphi^{|\eta|+1}(x, b_{\eta \smallfrown \langle i \rangle})\}_{i < \omega}$ form an inconsistent set.

One of the main cases of the Koponen conjecture from \cite{Kop16}, resolved in Theorem 6.2 of \cite{baldwin2024simple}, is that every simple theory with quantifier elimination in a finite relational language is supersimple. We can now show that this result holds at the level of $\exists\forall$-formulas. (This proof, again, uses the techniques of the main result this section. However, we would not be surprised if it was also possible to prove this corollary using an extremely tedious analysis of the proof of this simple-to-supersimple case of the Koponen conjecture.)

\begin{corollary}
    Let $T$ be theory with quantifier elimination in a finite relational language. Suppose $T$ is not supersimple, exhibited by the formulas $\varphi^{i}(x, y^{i})$. Then $T$ is non-simple, exhbited by some $\exists\forall$-formula in the $\varphi^{i}(x, y^{i})$\footnote{A $\exists\forall$-formula in a family of formulas is defined similarly to a $\exists\forall$-formula in a single formula, just starting with a Boolean combination of instances of formulas in this family instead of a Boolean combination of instances of the single formula}.

\end{corollary}

\begin{proof}

(sketch) Let $T_{0}$ be a reduct of $T$.  Note that $T$ is countably categorical. The theory  $T_{0}$ is then countably categorical, so by Fact \ref{countably categorical theories have model companions} it has a countably categorical model companion $T'$. It is not known whether the property of having quantifier elimination in a finite relational language is preserved under model companions (see Question 1.1 of \cite{bodirsky2025taking}), and this property is also not known to be preserved under reducts. However, $T$ at least has the weaker property that there are finitely many types of indiscernible sequences (over $\emptyset$) in any fixed finite variables. One can show that this is preserved under taking model companions and reducts, so $T'$ has this property as well. But it is even shown in Theorem 6.3 of \cite{baldwin2024simple} that, for any countably categorical theory with finitely many types of indiscernible sequences in any fixed finite variables, that theory must be either supersimple or non-simple. So $T'$, if we can show that it is not supersimple, will not be simple. We can then proceed as in the proof of Theorem \ref{theorem on implications between positively straightly definable properties in countably categorical theories}.

\end{proof}

\section{Open Questions}\label{open questions}

Prior to the results of this paper, it was known that all of the original properties from Shelah's classical classification-theoretic hierarchy were straightly definable, with the exception of $\mathrm{SOP}_{n}$ for integers $n \geq 4$. Our result, Theorem \ref{SOPn straightly definable}, that these properties are also straightly definable therefore completes the categorization.

However, as shown in the table at the end of Section \ref{New results on SOPn}, there remains one gap in the categorization of the original classification-theoretic properties. This gap was first noted in the following question of \cite{garcia2022model}:

\begin{question}
    Is the tree property, $\mathrm{TP}$, poset definable? 
\end{question}

It should also be noted that the fact that $\mathrm{SOP}_{n}$ is straightly definable and poset definable for all $n$ only applies to \textit{integer} values of $n$. The case of non-integer values of $n$, for which the properties $\mathrm{SOP}_{n}$ were introduced in \cite{mutchnik2025approximations}\footnote{Another family of properties defined in terms of quasimetrics, referred to as an extension of $\mathrm{SOP}_{r}$ to real values of $r$, was defined in a thesis appendix of \cite{hanson2020definability}; however, as Hanson shows, these coincide  for relatively straightforward reasons with the original properties $\mathrm{SOP}_{n}$ for $n$ an integer, so unlike the real-valued $\mathrm{SOP}_{r}$ hierarchy introduced by the second author of this paper, it's not plausible that these could actually be new properties.}, remains open: 

\begin{question}
    Are the properties $\mathrm{SOP}_{r}$ straightly definable for non-integer values of $r$? Are they poset definable?
\end{question}

To our knowledge, the proof that $\mathrm{SOP}_{n}$ is straightly definable and poset definable for integer values of $n$ does not straightforwardly extend to this case.

In fact, it remains open whether the properties $\mathrm{SOP}_{r}$, for non-integer values of $r$, are even new properties at all. We could make the following bold conjecture: $\mathrm{SOP}_{n}$ is straightly definable (or poset definable) \textit{if and only if} $n$ is an integer. We have now shown the ``if" direction; the ``only if" direction would therefore imply that the properties $\mathrm{SOP}_{r}$ for non-integer values of $r$ really are new classification-theoretic properties. Moreover, proving the ``only if" direction would reshape our understanding of the nature of the newness problem for non-integral $\mathrm{SOP}_{r}$. A reasonable intuition may suggest that, if it's the case that the properties $\mathrm{SOP}_{r}$ for non-integer values of $r$ are new properties, this would demonstrate that the integers do not have a special status within the strict order property hierarchy. However, showing that only the properties in the integer-valued $\mathrm{SOP}_{n}$ hierarchy are straightly definable would suggest, contrary to this intuition, that the real-valued $\mathrm{SOP}_{r}$ hierarchy does introduce new properties \textit{precisely because} the integers have a special status.

\textbf{Acknowlegdements:} We would like to thank Mirna Džamonja, Rosario Mennuni and Erik Walsberg for helpful comments on earlier versions of this paper. The second author would like to thank Alex Kruckman for first drawing his attention to the problem of straight definability of $\mathrm{SOP}_n$ for integers $n \geq 4$. 

\appendix

\section{Shelah's ``straight definition" of $\mathrm{SOP}_{n}$}\label{shelah ``straight definition"}

 Shelah, aiming to give a straight definition of $\mathrm{SOP}_{n}$ for integers $n \geq 3$, claims that a theory has $\mathrm{SOP}_{n}$ if and only if it satisfies the following:\footnote{While he actually states that $\mathrm{NSOP}_{n}$ is equivalent to this property, this is most likely just a typographical error.}

    There is a formula $\varphi(x, y)$ and an indiscernible sequence $\{a_{i}\}_{i < \omega}$ such that, for each $0=i_{0} < \ldots < i_{n} = \omega$, and $\varphi(x, y)^{\mathrm{if}(\mathrm{\ell \: even})}$ defined to be $\varphi(x, y)$ if $\ell$ is even and $\neg\varphi(x, y)$ otherwise, $\{\varphi(x, a_m)^{\text{if}(\ell\text{ even})} : \ell < n, m \in [i_\ell,i_{\ell+1})\}$ is consistent, but if $0=i_{0} < \ldots < i_{n} < \omega$, $\{\varphi(x, a_{i_\ell})^{\text{if}(\ell\text{ even})} : \ell \leq n)\}$ is inconsistent.

    By extracting indiscernible sequences, this is a straightly definable property. However, we show that this proposed straight definition for $\mathrm{SOP}_{n}$ fails. We can change the index set, replacing $\{a_{i}\}_{i < \omega}$ with $\{a_{i}\}_{i \in \mathbb{Q}_{\geq 0}}$. Thus, for $n \geq 3$, the proposed straight definition of $\mathrm{SOP}_{n}$ will fail if there is a theory with $\mathrm{SOP}_{n}$ with the following property, which contradicts the proposed straight definition of $\mathrm{SOP}_{n}$:

    Let $\{a_{i}\}_{i < \omega}$ be an indiscernible sequence with $b =: a_{0} \cap a_{1} = \bigcap a_{i}$, and let $\varphi(x, y)$ be a formula and $p(x) \in S(b)$ a complete type such that both $p(x) \cup \{\varphi(x, a_{i})\}_{ i < \omega}$ is consistent, and $p(x) \cup \{\neg\varphi(x, a_{i})\}_{ i < \omega}$ is also consistent. (Note that this follows from the consistency condition with $\{a_{i}\}_{i \in \mathbb{Q}_{\geq 0}}$ in Shelah's proposed straight definition, where $p(x)$ is the type over $b$ of a realization for this consistency condition.) Then $\{\varphi(x, a_{i})^{\text{if}(i \text{ even})} \}_{ i < \omega}$ is consistent. (This contradicts the inconsistency condition.)

    We first show that the model companion of the theory of triangle-free graphs is an $\mathrm{SOP}_{3}$ theory with this property contradicting Shelah's proposed straight definition. Let $\{a_{i}\}_{i < \omega}$ be an indiscernible sequence satisfying the hypotheses for $\varphi(x, y)$, with $b =: a_{0} \cap a_{1} = \bigcap a_{i}$ and $p(x) \in S(b)$. Either $\varphi(x, a_{0})$ or $\neg\varphi(x, a_{0})$ will be satisfied by some $c \models p(x)$ with $c \cap a_{0} \subset b$ and with no edges between $c \backslash b$ and $a_{0} \backslash b$. (This is just because there does exist some $c \models p(x)$ with $c \cap a_{0} \subset b$ and no edges between $c \backslash b$ and $a_{0} \backslash b$.) Suppose without loss of generality that $\neg\varphi(x, a_{0})$ is so satisfied. Choose $d \models p(x) \cup \{\varphi(x, a_{i})\}_{ i < \omega}$ over which $\{a_{i}\}_{i< \omega}$ is indiscernible; then $d \cap a_{i} = b$. By removing edges we may find $d'$ with $d'a_{i} \equiv da_{i}$ for $i$ even, and $d'a_{i} \equiv c a_{0}$ for $i $ odd. Then $d' \models \{\varphi(x, a_{i})^{\text{if}(i \text{ even})} \}_{ i < \omega} $, as desired.

    The verification for $\mathrm{SOP}_{n}$ for $n \geq 4$ uses the same idea. We sketch this argument, verifying the property contradicting Shelah's proposed straight definition for the model companion of the theory of $\leq n$-cycle-free directed graphs, which has $\mathrm{SOP}_{n}$. Again let $\{a_{i}\}_{i < \omega}$ be an indiscernible sequence satisfying the hypotheses for $\varphi(x, y)$, with $b =: a_{0} \cap a_{1} = \bigcap a_{i}$ and $p(x) \in S(b)$. Let $n= 2k$ or $n = 2k-1$ (so $k$ is the minimum such that any two points have a path of length at most $k$). Say $c \models p(x)$ is \textit{freely amalgamated} with $a_{i}$ over $b$ if $c \cap a_{i} \subset b$ and the minimum length of a path (in the ambient model) between a point of $c \backslash b$ and a point of $a_{i} \backslash b$ is the minimum length of a path \textit{each of whose edges are in $cb$ or $a_{i}b$} between those points, or $k$, whichever is lower. Additionally, for $k' \leq k$, say $c$ is $k'$-\textit{freely amalgamated} with $a_{i}$ over $b$ if $c \cap a_{i} \subset b$ and the minimum length of a path (in the ambient model) between a point of $c \backslash b$ and a point of $a_{i} \backslash b$ is the minimum length of a path \textit{each of whose edges are in $cb$ or $a_{i}b$} between those points, or $k'$, whichever is lower.

    Then either $\varphi(x, a_{0})$ or $\neg\varphi(x, a_{0})$ is satisfied by some $c \models p(x)$ freely amalgamated with $a_{0}$ over $b$. Without loss of generality we may assume that $\neg \varphi(x, a_{0})$ is so satisfied. Suppose for a contradiction that $\{\varphi(x, a_{i})^{\text{if}(i \text{ even})} \}_{ i < \omega}$ were inconsistent. Again choose $d \models p(x) \cup \{\varphi(x, a_{i})\}_{ i < \omega}$ over which $\{a_{i}\}_{i< \omega}$ is indiscernible; then, again $d \cap a_{i} = b$. We first show that $d$ can be chosen freely amalgamated with $a_{i}$ over $b$. We know that $d$ is at least $1$-freely amalgamated with $a_{i}$ over $b$, so by induction it sufficies to assume that $d$ is $k'$-freely amalgamated with $a_{i}$ over $b$ for $k' < k$, and then show that we can choose $d$ to be $k'+1$-freely amalgamated with $a_{i}$ over $b$. Starting with $d \models p(x) \cup \{\varphi(x, a_{i})\}_{ i < \omega}$ over which $\{a_{i}\}_{i< \omega}$ is indiscernible and with $d$ $k'$-freely amalgamated with $a_{i}$ over $b$, we can find $d'$ such that $d'a_{i} \equiv da_{i}$ for $i$ even, and $d'$ is $k'+1$-freely amalgamated with $a_{i}$ over $b$ for $i$ odd. By the supposition that $\{\varphi(x, a_{i})^{\text{if}(i \text{ even})} \}_{ i < \omega}$ is inconsistent and indiscernibility, among the odd $i$, $d' \models \varphi(x, a_{i})$ for all but finitely many $i$. So by indiscernibility and extracting an indiscernible sequence, we may find $d'' \models p(x) \cup \{\varphi(x, a_{i})\}_{ i < \omega}$ over which $\{a_{i}\}_{i< \omega}$ is indiscernible and with $d''$ $k'+1$-freely amalgamated with $a_{i}$ over $b$, as desired.

    So we now have $d \models p(x) \cup \{\varphi(x, a_{i})\}_{ i < \omega}$ over which $\{a_{i}\}_{i< \omega}$ is indiscernible, and with $d$ freely amalgamated with $a_{i}$ over $b$. At this point we may find $d'$ with $d'a_{i} \equiv da_{i}$ for $i$ even, and $d'a_{i} \equiv c a_{0}$ for $i $ odd. Then $d' \models \{\varphi(x, a_{i})^{\text{if}(i \text{ even})} \}_{ i < \omega} $, contradicting our supposition that this set of formulas was inconsistent.

\section{ $\mathrm{SOP}_{1}$, $ \mathrm{SOP}_{2}$ and formula complexity}\label{formula complexity sop1 sop2}

Let $\varphi(x, y)$ be a $\mathrm{SOP}_{1}$ formula, and $R_{n}(a_{1}, \ldots, a_{n}) =: (a_{1}, \ldots, a_{n}) \models \exists x \bigwedge_{i=1}^{n}\varphi(x, y_{i})$. Remark 3.15 of \cite{mutchnik2023properties} states a partial argument, which it acknowledges to be incomplete at a specific step, that the proof of \cite{Mutchnik2026nsop2} that every theory with $\mathrm{SOP}_{1}$ has $\mathrm{SOP}_{2}$ can in fact be analyzed to show that there must be some formula $\varphi'(x', y')$, consisting of a Boolean combination of the $R_{n}$, that exhibits $\mathrm{SOP}_{2}$. In this appendix, we state why the step acknowledged to be incomplete in that remark can be fixed, giving us our desired $\mathrm{SOP}_{2}$ formula $\varphi'(x', y')$.

    As stated in Remark 3.15 of \cite{mutchnik2023properties}, if $\varphi'(x, y)$ as desired does not exist, we can prove, \textit{in the quantifier-free context in the formulas $R_{n}$}, Kim’s
lemma for Morley sequences in a canonical coheir, symmetry for Conant-independence, and the weak
independence theorem for Conant-independence as in the original argument  that $\mathrm{SOP}_{1}$ is equal to $\mathrm{SOP}_{2}$. Having proven these, the part where the argument, as stated in that remark, could not be continued, was in inductively constructing a configuration $\{b^{1}_{i}b_{i}^{2}\}_{i < \omega}$ with the following properties for any fixed strong canonical coheir $q(x)$ over some model $M$:

(1) For $J_{n}$ the sequence beginning with $b_{i}^{2}$ for $i < n$ and then continuing with $b_{i}^{1}$ for $i \geq n$, $J_{n}$ is a strong canonical Morley sequence in $q(x)$ (so, by Kim’s
lemma for Morley sequences in a canonical coheir, is an $R_{j}$-clique for each $j < \omega$).

(2) For $i \leq j$, $b_{i}^{1} b_{j}^{2}\equiv^{\{R^{n}\} - \mathrm{qftp}}_{M}  b_{0}b_{1}$

(3) $b^{1}_{1} \ldots b^{1}_{n} \ldots \ind^{(K^{*})^{\{R_{n}\}-\mathrm{qf}}}_{M}b^{2}_{1} \ldots b^{2}_{n} \ldots$.

Here $\ind^{(K^{*})^{\{R^{n}\}-\mathrm{qf}}}$ is to be read as (coheir) Conant-independence as in the original argument, but \textit{in the quantifier-free sense with respect to the $R_{n}$}, while $\equiv^{\{R_{n}\} - \mathrm{qftp}}$ denotes satisfying the same quantifier-free type with respect to the $R_{n}$. The issue with continuing the argument stated in that remark is as follows: consider $\{b^{1}_{i}b_{i}^{2}\}_{i \leq n}$ already constructed, and attempt to add on $b^{1}_{n+1}$ and then $b^{2}_{n+1}$ as in the original argument. There is no problem in showing that adding on $b^{1}_{n+1} \models q(x)|_{M \{b^{1}_{i}b_{i}^{2}\}_{i \leq n} } $ preserves the truncations of the properties (1)-(3). However, the remark states that, because the weak independence theorem for Conant-independence only applies in the quantifier-free context in $R_{\infty}$, after adding on $b^{1}_{n+1}$ using that theorem as in the original argument, (1) is not preserved. Instead, we are left with only a configuration  $b^{1}_{n+1} \models q(x)|_{M \{b^{1}_{i}b_{i}^{2}\}_{i \leq n+1} } $ satisfying the truncations of the following properties:

(1$'$) For $J_{n}$ the sequence beginning with $b_{i}^{2}$ for $i < n$ and then continuing with $b_{i}^{1}$ for $i \geq n$, $J_{n}$ \textit{satisfies the same quantifier-free type over $M$ in the $R_{j}$} as a strong canonical Morley sequence in $q(x)$ (so, by Kim’s
lemma for Morley sequences in a canonical coheir, is an $R_{j}$-clique for each $j < \omega$).

(2$'$) Same as (2).

(3$'$) Same as (3).

However, if we can inductively build a configuration $\{b^{1}_{i}b_{i}^{2}\}_{i < \omega}$ satisfying (1$'$)-(3$'$), rather than (1)-(3), this would still complete the argument that the formula $\varphi'(x, y)$ as desired exists. So suppose $\{b^{1}_{i}b_{i}^{2}\}_{i \leq n}$ satisfying the truncations of (1$'$)-(3$'$) is already constructed; we need to show that adding on $b^{1}_{n+1} \models q(x)|_{M \{b^{1}_{i}b_{i}^{2}\}_{i \leq n} } $, and then $b^{2}_{n+1} $ using (the quantifier-free version in $R^{n}$ of) the weak independence theorem as in the original argument, the truncations of (1$'$)-(3$'$) are preserved. If we are able to show that adding on $b^{1}_{n+1}$ preserves these properties, then, on the basis of what we already know, adding on $b^{2}_{n+1} $ will preserve these properties: the issue with the weak independence theorem only applying at the quantifier-free level is that (1) is not preserved upon adding $b^{2}_{n+1} $, but there is no problem preserving (1$'$). So the issue is showing, starting with $\{b^{1}_{i}b_{i}^{2}\}_{i \leq n}$, that adding $b^{1}_{n+1} \models q(x)|_{M \{b^{1}_{i}b_{i}^{2}\}_{i \leq n} } $ preserves (1$'$)-(3$'$). And preservation of (2$'$) and (3$'$) by adding $b^{1}_{n+1} \models q(x)|_{M \{b^{1}_{i}b_{i}^{2}\}_{i \leq n} } $ will be as in the original argument.\footnote{More technically, it will be as in the original argument at the quantifier-free level in the $R_{n}$. Specifically, to prove the analogue of Claim 6.2 of \cite{Mutchnik2026nsop2} to show that (3) is preserved, we must assume that $\{b^{1}_{i}\}_{i \leq n}$ begins a canonical coheir Morley sequence over $M$ indiscernible over $M\{b^{2}_{i}\}_{i \leq n}$. But using $b^{1}_{1} \ldots b^{1}_{n}  \ind^{(K^{*})^{\{R_{n}\}-\mathrm{qf}}}_{M}b^{2}_{1} \ldots b^{2}_{n}$, we may replace $\{b^{1}_{i}b_{i}^{2}\}_{i \leq n}$ so that this is the case without changing the quantifier-free type over $M$ in the $R_{n}$.} So the part where a new observation is needed is showing that adding $b^{1}_{n+1} \models q(x)|_{M \{b^{1}_{i}b_{i}^{2}\}_{i \leq n} } $ preserves (1$'$).

We give this new observation. All we need is to prove the following claim:

\begin{claim}
    Let $p(x)$ be a global $M$-coheir, and let $\varphi(x, y)$ be a quantifier-free formula in the $R^{n}$. Suppose that (for $|b| = |y|$) $b$ and $b'$ satisfy the same quantifier-free type over $M$ in the $R^{n}$. Then $\varphi(x, b) \in p(x)$ if and only if $\varphi(x, b') \in p(x)$.

    Therefore, for $q(x)$ our strong canonical coheir, if $b_{0} \ldots b_{n}$ satisfies the same quantifier-free type over $M$ in the $R^{n}$ as the first $n+1$ terms of a strong canonical Morley sequence over $M$ in $q(x)$, for $b_{n+1} \models q(x)|_{Mb_{0} \ldots b_{n}}$,  $b_{0} \ldots b_{n+1}$ will satisfy the same quantifier-free type over $M$ in the $R^{n}$ as the first $n+2$ terms of a strong canonical Morley sequence over $M$ in $q(x)$.
\end{claim}

\begin{proof}

The statement of the main clause of the claim is just a refinement of the well-known fact that an $M$-coheir is an $M$-invariant type; while the conclusion should not be true for $M$-invariant types in general, the standard proof of the fact that \textit{coheirs} are invariant actually shows this stronger fact.  Specifically, suppose without loss of generality that $\varphi(x, b) \in p(x)$, $\neg\varphi(x, b') \in p(x)$ and let $a \models p(x)|_{Mbb'}$, so $\models \varphi(a, b) \wedge \neg\varphi(a,b')$. By finite satisfiability, there is $m \in M$ such that 
$\models \varphi(m, b) \wedge \neg\varphi(m,b')$. Because $\varphi(x, y)$ is assumed to be a quantifier-free formula in the $R^{n}$, this contradicts that $b$ and $b'$ satisfy the same quantifier-free type over $M$ in the $R^{n}$.

\end{proof}

So we have shown that there is a formula $\varphi'(x', y')$ exhibiting $\mathrm{SOP}_{2}$ which is a quantifier-free formula in the $R_{n}$. In particular $\varphi'(x', y')$ will be a Boolean combination of existential and universal formulas in $\varphi(x, y)$. 

\section{An additional remark on $\mathrm{PM}^{(2)}$}

As a supplement to the main results of this paper, we develop the theory of a particular model-theoretic patterning property specifically motivated by the general notion of a positively straightly definable property of theories. This property $\mathrm{PM}^{(2)}$  or \textit{$2$-positive maximality}, was defined in Definition 6.2 of \cite{bailetti2024walk}. It is related to a special kind of positive pattern, the \textit{maximal consistency patterns} mentioned earlier defined in Definition 3.1 of \cite{garcia2022model}. (Note that the word ``maximal" is used in two different senses between these definitions.)

\begin{definition}

(1) A \textit{maximal consistency pattern} is a positive pattern $(\mathcal{C}, \mathcal{I})$ such that:

\begin{itemize}
    \item each $(Z^{+}, \emptyset) \in \mathcal{I}$ has $|Z^{+}|=2$
    \item the sets $Y^{+} \subset \omega$ with $(Y^{+}, \emptyset) \in \mathcal{C}$ are exactly those sets maximal with the property of not containing any $Z^{+}$ with  $(Z^{+}, \emptyset) \in \mathcal{I}$
    \item every $i \in \omega$ belongs to at least two distinct pairs $Z^{+}$ with $(Z^{+}, \emptyset) \in \mathcal{I}$.\footnote{While this last condition may not matter much for our purposes, the reason why Garcia and Mennuni include this condition is to use it in their proof that every positively straightly definable property defined by a maximal consistency pattern is poset definable.}
\end{itemize}

(2) A formula is $2$-\textit{positively maximal} or has $\mathrm{PM}^{(2)}$ if it has every positively straightly definable property defined by a maximal consistency pattern. Equivalently, it has every positively straightly definable property definable by a positive pattern $(\mathcal{C}, \mathcal{I})$ such that each $(Z^{+}, \emptyset) \in \mathcal{I}$ has $|Z^{+}|=2$.\footnote{This is because, for every positively straightly definable property $P$ defined by a positive pattern $(\mathcal{C}, \mathcal{I})$ with $|Z^{+}|=2$ for each $(Z^{+}, \emptyset) \in \mathcal{I}$, $P$ is implied by some positively straightly definable property defined by a maximal consistency pattern.} A theory has $\mathrm{PM}^{(2)}$ if it has a formula with $\mathrm{PM}^{(2)}$; otherwise the theory is $\mathrm{NPM}^{(2)}$
    
\end{definition}

So this definition is analogous to the property of \textit{straight maximality} defined in Definition 5.20 of \cite{shelah_what_2000}, stating that there is a formula with every straightly definable property.  While Bailetti also defines $\mathrm{PM}^{(k)}$ for all $k \geq 2$, the proeprty $\mathrm{PM}^{(2)}$ appears special among these; for example, it is open whether $\mathrm{NPM}^{2} \cap \mathrm{NSOP} = \mathrm{NSOP}_{3}$.

By Fact \ref{positive straight definition of sop3}, $\mathrm{SOP}_{3}$ is positively straightly definable by a maximal consistency pattern.  Moreover, $\mathrm{TP}_{2}$ (defined in Remark \ref{preservation of sopn} above) is positively straightly definable by a maximal consistency pattern. So their negations, $\mathrm{NSOP}_{3}$ and $\mathrm{TP}_{2}$, imply $\mathrm{NPM}^{(2)}$. In \cite{mutchnik2023properties}, it is shown that $\mathrm{NSOP}_{3}$ theories exhbit some phenomena similar to those known or conjectured for $\mathrm{NTP}_{2}$ theories. We focus on one phenomenon in particular, the relationship between \textit{internally $\mathrm{NSOP}_{1}$} and \textit{co-$\mathrm{NSOP}_{1}$} types in these theories:

\begin{definition}
    (1) Let $p(x)$ be an $n$-type over $M$. Let $\mathcal{L}_{p}$ contain an $m$-ary relation symbol $R_{\varphi}$ for each formula $\varphi(x_{1}, \ldots, x_{m}) \in L(M)$ with $|x_{i}| = n$ for $i \leq m$. Let $\mathcal{M}_{p}$ be the $\mathcal{L}_{p}$-structure with domain $p(\mathbb{M}^{n})$ and with $R_{\varphi}(p(\mathbb{M}^{n})^{m}) = \varphi(\mathbb{M}^{mn}) \cap p(\mathbb{M}^{n})^{m}$. Then $p(x)$ is \textit{internally $\mathrm{NSOP}_{1}$} if the theory of $\mathcal{M}_{p}$ is $\mathrm{NSOP}_{1}$.

    (2) A type $p(x)$ over a model $M$ is \textit{co}-$\mathrm{NSOP}_{1}$ if there does not exist a formula $\varphi(x, y) \in L(M)$ and tuples $\{b_{\eta}\}_{\eta \in 2^{<\omega}}$, $b_{\eta} \subset p(\mathbb{M})$ such that $\{\varphi(x, b_{\sigma \upharpoonleft n})\}_{n < \omega}$ is consistent for any $\sigma \in 2^{\omega}$, but for any $\eta_{2} \unrhd \eta_{1} \smallfrown \langle 0\rangle$, $\{\varphi(x, b_{\eta_{2}}), \varphi(x, b_{\eta_{1} \smallfrown \langle 1\rangle})\}$ is inconsistent.
\end{definition}

The following, for $\mathrm{NSOP}_{3}$ theories, is Theorem 1.1 of \cite{mutchnik2023properties}; for $\mathrm{NTP}_{2}$ theories it is a straightforward generalization of Theorem 6.17 of \cite{chernikov2014theories}.

\begin{fact}
    Let $T$ be $\mathrm{NSOP}_{3}$ or $\mathrm{NTP}_{2}$. Then every internally $\mathrm{NSOP}_{1}$ type is co-$\mathrm{NSOP}_{1}$.
\end{fact}

Generalizing this fact, the goal of the rest of this section will be to sketch arguments generalizing this fact to apply not just to $\mathrm{NSOP}_{3}$ or $\mathrm{NTP}_{2}$ theories, but to all $\mathrm{NPM}^{(2)}$ theories:

\begin{proposition}\label{internally nsop1 types co-nsop1 in nmp2 theory}
    Let $T$ be $\mathrm{NPM}^{(2)}$. Then every internally $\mathrm{NSOP}_{1}$ type is co-$\mathrm{NSOP}_{1}$.
\end{proposition}

This proposition suggests that the fact that both $\mathrm{SOP}_{3}$ and $\mathrm{TP}_{2}$ are positively straightly definable properties defined by maximal consistency patterns, so $\mathrm{NSOP}_{3}$ and $\mathrm{NTP}_{2}$ imply $\mathrm{NPM}^{(2)}$, is the underlying explanation for why $\mathrm{NSOP}_{3}$ and $\mathrm{NTP}_{2}$ both exhibit this phenomenon. Common phenomena within $\mathrm{NSOP}_{3}$ and $\mathrm{NTP}_{2}$ that can be explained by $\mathrm{NPM}^{(2)}$ are not limited to this one; in an upcoming preprint of the second author, we will discuss a common phenomenon related to cycle removal that, as we will remark, actually applies to all $\mathrm{NPM}^{(2)}$ theories.

We will mostly refer to the arguments for Theorem 1.1 from \cite{mutchnik2023properties}, but will prove a specific combinatorial claim, analogous to Proposition 3.12 from that paper. We briefly discuss how that combinatorial claim fits into those arguments. Consider sequences of hypergraphs $R_{\infty}=(V,\{R_{n}\}_{n < \omega})$ on a common set of vertices $V$, where $R_{n}$ is an $n$-ary edge relation consisting of $n$-tuples of distinct elements which is permutation invariant, and each edge of $R_{n}$ is an $R_{m}$-clique for $m \leq n$. Every positively straightly definable property has an analogue for properties of hypergraph sequences, where consistency of $n$ instances of a formula is replaced by the presence of an $R_{n}$-edge between $n$ vertices, and inconsistency of $n$ instances of a formula is replaced by the absence of an $R_{n}$-edge between $n$ vertices. So $\mathrm{SOP}_{1}$ has a hypergraph analogue $\mathrm{MSOP}_{1}$ (where the ``$\mathrm{M}$" is after \cite{Mal10}), $\mathrm{TP}_{2}$ has $\mathrm{MTP}_{2}$ (there called an \textit{$(\omega, \omega, 1)$-array}, as in Malliaris's terminology), $\mathrm{SOP}_{3}$ has $\mathrm{MSOP}_{3}$ (there called the \textit{compatible order property} as in Malliaris), and so on. The second author's proof that internally simple types are co-simple in $\mathrm{NSOP}_{3}$ theories reduces to showing that, for $R_{\infty}=(V,\{R_{n}\}_{n < \omega})$ any definable hypergraph sequence in an $\mathrm{NSOP}_{1}$ theory (particularly, in $\mathcal{M}_{p}$ in the internally $\mathrm{NSOP}_{1}$ type $p$), if $R_{\infty}$ has $\mathrm{MSOP}_{1}$, $R_{\infty}$ has $\mathrm{MSOP}_{3}$. The proof for $\mathrm{NTP}_{2}$ involves doing the same, but with $\mathrm{MTP}_{2}$ instead of $\mathrm{MSOP}_{3}$. Thus, to prove Proposition \ref{internally nsop1 types co-nsop1 in nmp2 theory}, it will suffice to define $\mathrm{MPM}^{(2)}$ and show that $\mathrm{MSOP}_{1}$ implies $\mathrm{MPM}^{(2)}$ in hypergraph sequences definable in $\mathrm{NSOP}_{1}$ theories.

We first restate the definition of $\mathrm{MSOP}_{1}$, then define $\mathrm{MPM}^{(2)}$:

\begin{definition}
        Let $R_{\infty}=(V,\{R_{n}\}_{n < \omega})$ be a sequence of hypergraphs on a common set of vertices $V$, where $R_{n}$ is an $n$-ary edge relation consisting of $n$-tuples of distinct elements which is permutation invariant, and each edge of $R_{n}$ is an $R_{m}$-clique for $m \leq n$. Then $R_{\infty}$ has

        (1)  $\mathrm{MSOP}_{1}$ if (in some model of the theory of $R_{\infty}=(V,\{R_{n}\}_{n < \omega})$) there are vertices $\{b_{\eta}\}_{\eta \in 2^{<\omega}}$ such that $\{ b_{\sigma \upharpoonleft n}\}_{n < \omega}$ is a clique for each of the $R_{n}$ for any $\sigma \in 2^{\omega}$, but for any $\eta_{2} \unrhd \eta_{1} \smallfrown \langle 0\rangle$, $\{b_{\eta_{2}},  b_{\eta_{1} \smallfrown \langle 1\rangle}\}$ is a $2$-anticlique.

        (2) $\mathrm{MPM}^{(2)}$ if, for $R'_{\infty} =: \{R'_{n}\}_{n < \omega}$ a hypergraph sequence on some vertices $V'$ such that $(V', R'_{2})$ is the random graph, and the sets of $n$ vertices that are $R'_{n}$-edges are exactly the $R_{2}$-cliques, some model of the theory of $R_{\infty}=(V,\{R_{n}\}_{n < \omega})$ embeds $R'_{\infty} $ as a sequence of induced subhypergraphs.
\end{definition}

To prove Proposition \ref{internally nsop1 types co-nsop1 in nmp2 theory}, it will be enough to prove the following claim:

\begin{claim}
    Let $T$ be $\mathrm{NSOP}_{1}$, and let $R_{\infty}=(V,\{R_{n}\}_{n < \omega})$ be a sequence of hypergraphs on $V$ with $V$, $R_{n}$ definable in $T$. Then if $R_{\infty}$ has $\mathrm{MSOP}_{1}$, $R_{\infty}$ has $\mathrm{MPM}^{(2)}$.
\end{claim}

\begin{proof}

We freely make use of facts about Kim-independence in $\mathrm{NSOP}_{1}$ theories; see \cite{kaplanramsey2017kim} for details. As in the argument that $\mathrm{MSOP}_{1}$ implies $\mathrm{MSOP}_{3}$ or $\mathrm{MTP}_{2}$ in an $\mathrm{NSOP}_{1}$ theory, by $\mathrm{MSOP}_{1}$ there is a model $M$, and $M$-invariant Morley sequences $\{a_{i}\}_{i < \omega}$ in the $M$-invariant type $p(x)$, $\{b_{i}\}_{i < \omega}$ in the $M$-invariant type $q(x)$, $p(x)|_{M}=q(x)|_{M}$ with $\{a_{i}\}_{i < \omega}$ a clique for each of the $R_{n}$, and $\{b_{i}\}_{i < \omega}$ an anticlique for each of the $R_{n}$ for $n \geq 2$ (the ``Kim's lemma" characterization of $\mathrm{MSOP}_{1}$). Let $(\upsilon, \{\rho_{n}\}_{n \leq |\upsilon|})$ be sequence of hypergraphs on the finitely many vertices $\upsilon=\{v_{1}, \ldots, v_{|\upsilon|}\}$ such that the $\rho_{n}$-edges are exactly the $\rho_{2}$-cliques; it suffices to show $(V,\{R_{n}\}_{n < \omega})$ embeds $(\upsilon, \{\rho_{n}\}_{n < \omega})$ as a sequence of induced subhypergraphs. We will show by induction that, for $k \leq |v|$, there are $M$-mutually indiscernible\footnote{i.e, each $I_{i}$ is indiscernible over $MI_{\neq i}$, Definition 1.1 of \cite{chernikov2014theories}} invariant Morley sequences $I_{1}, \ldots I_{k}$ over $M$ such that, for some/any choice of $a_{i} \in I_{i}$, the sequence $a_{1}, \ldots, a_{k}$ is such that $a_{i} \mapsto v_{i}$ is an embedding as a sequence of induced subhypergraphs, and every subsequence of the $a_{1}, \ldots, a_{k}$ which forms an $R_{n}$-edge/$R_{2}$-clique starts an $M$-invariant Morley sequence in $p(x)$; once we have found $I_{1}, \ldots I_{|\upsilon|}$, we will have obtained an even stronger version of our desired instance of $(\upsilon, \{\rho_{n}\}_{n < \omega})$.

Let $k < |\upsilon|$ and suppose $I_{1}, \ldots, I_{k}$ as above are already constructed; we show that we can add $I_{k+1}$. Let $S = \{i \leq k: \rho_{2}(v_{k+1}, v_{i})\}$. It suffices to find an $M$-invariant Morley sequence $I_{k+1}$ such that $I_{1}, \ldots I_{k+1}$ are mutually indiscernible, each term of $I_{k+1}$ satisfies $p(x)|_{M\{I_{i}\}_{i \in S}}$, and for each $i \notin S$, each term of $I_{k+1}$ satisfies $q(x)|_{MI_{i}}$. Fix some arbitrary choice of $M$-invariant type $r(x)$ extending $p(x)|_{M} = q(x)|_{M}$. Note that $p(x)$, as an $M$-invariant type, does not Kim-fork over $M$. So by symmetry and the chain condition for Kim-independence in $\mathrm{NSOP}_{1}$ theories, there is a type $\tilde{p}(x)\in S(M\{I_{i}\}_{i \in S})$, which does not Kim-fork over $M$, realized by $M$-invariant Morley sequences in $r(x)$ each term of which satisfies $p(x)|_{M\{I_{i}\}_{i \in S}}$. Moreover, for the same reasons applied to $q(x)$, for each $i \in \{1, \ldots, k\}\backslash S$ there is a type $\tilde{q}_{i}(x) \in S(MI_{i})$, which also does not Kim-fork over $M$, realized by $M$-invariant Morley sequences in $r(x)$ each term of which satisfies $q(x)|_{MI_{i}}$. Now because each $I_{i}$ for $i \leq k$ is an $M$-invariant Morley sequence indiscernible over $I_{\neq i}$, by symmetry and the chain condition for Kim-independence, $I_{i} \ind_{M}^{K} I_{\neq i}$. In particular, for $i_{1} < \ldots <i_{\ell}$ the increasing enumeration of $\{1, \ldots, k\} \backslash S$, $\{I_{i}\}_{i \in S}, I_{i_{1}}, \ldots, I_{i_{j}}, \ldots I_{i_{\ell}} $ form a Kim-independent sequence over $M$: for $j \leq \ell$, $I_{i_{j}} \ind_{M}^{K} \{I_{i}\}_{i \in S}I_{i_{1}}, \ldots I_{i_{j-1}}$. So by repeated applications of the independence theorem for Kim-independence, we can find $I_{k+1} \models \tilde{p}(x) \cup \bigcup_{i \leq \ell} \tilde{q}_{i_{\ell}}(x)$, and extracting mutually indiscernible sequences using Lemma 1.2.1 of \cite{chernikov2014theories}\footnote{Note that the statement of this lemma in that paper technically applies to sets of sequences of a certain infinite length; this would not be needed in our case, because the conditions we would like our mutually indiscernible sequences to satisfy are type-definable. But to apply that lemma as stated, we could just use compactness to find sequences of the correct length satisfying these conditions, then apply the lemma to those sequences.}, we can assume $I_{1}, \ldots, I_{k+1}$ are mutually indiscernible, as desired.

\end{proof}

Note that Example 3.9 from \cite{mutchnik2023properties}, the model companion of the theory of triangle-free tripartite graphs, which was given as an example of an $\mathrm{NSOP}_{4}$ theory where not every internally $\mathrm{NSOP}_{1}$ type is co-$\mathrm{NSOP}_{1}$, is $\mathrm{NPM}^{(3)}$ (i.e., satisfies the negation of $\mathrm{PM}^{(3)}$, where $\mathrm{PM}^{(3)}$ means that a theory satisfies all of the positively straightly definable properties defined by positive patterns $(\mathcal{C}, \mathcal{I}$) where for an inconsistency condition $(Z^{+}, \emptyset) \in \mathcal{I}$, $|Z^{+}|=3$, as in Definition 6.2 of \cite{bailetti2024walk}; $\mathrm{PM}^{(k)}$ is defined similarly.) So the proposition we just proved about $\mathrm{NSOP}_{1}$ types being co-$\mathrm{NSOP}_{1}$ in $\mathrm{NPM}^{(2)}$ theories really does require $\mathrm{NPM}^{(2)}$ in the strict sense, and does not extend to $\mathrm{NPM}^{(k)}$ for $k > 2$.
\bibliographystyle{apalike}
\bibliography{bibliography}

\end{document}